\documentclass{amsart} 

\usepackage{amsthm, amssymb, amsmath, amscd}
\usepackage{eurosym}
\usepackage{amsfonts}
\usepackage{amsmath}
\usepackage{setspace}
\usepackage{bbm}
\usepackage{graphicx}

\usepackage[all,cmtip]{xy}

\usepackage{indentfirst}

\newtheorem{theorem}{Theorem}[section]
\newtheorem*{theorem*}{Theorem}

\newtheorem{cor}[theorem]{Corollary}

\newtheorem{lemma}[theorem]{Lemma}

\newtheorem{prop}[theorem]{Proposition}

\theoremstyle{definition}
\newtheorem{define}[theorem]{Definition}

\newtheorem{remark}[theorem]{Remark}

\DeclareFontFamily{OT1}{pzc}{}
\DeclareFontShape{OT1}{pzc}{m}{it}{<-> s * [1.1] pzcmi7t}{}
\DeclareMathAlphabet{\mathpzc}{OT1}{pzc}{m}{it}

\newcommand{\ch}{\mathrm{ch}}
\newcommand{\ind}{\mathrm{ind}}

\newcommand{\field}[1]{\mathbb{#1}}

\newcommand{\C}{\mathbbm{C}}
\newcommand{\Z}{\mathbbm{Z}}

\begin{document}
\title[Realizing the analytic surgery group geometrically: Part III]{Realizing the analytic surgery group of Higson and Roe geometrically\\
Part III: Higher invariants}

\author{Robin J. Deeley, Magnus Goffeng}

\date{\today}

\begin{abstract}
We construct an isomorphism between the geometric model and Higson-Roe's analytic surgery group, reconciling the constructions in the previous papers in the series on ``\emph{Realizing the analytic surgery group of Higson and Roe geometrically}" with their analytic counterparts. Following work of Lott and Wahl, we construct a Chern character on the geometric model for the surgery group; it is a ``delocalized Chern character", from which Lott's higher delocalized $\rho$-invariants can be retrieved. Following work of Piazza and Schick, we construct a geometric map from Stolz' positive scalar curvature sequence to the geometric model of Higson-Roe's analytic surgery exact sequence. 
\end{abstract}

\maketitle

\section*{Introduction}
This paper is the last in a series of three whose topic is the construction of a geometric (i.e., Baum-Douglas type) realization of the analytic surgery group of Higson and Roe \cite{HRSur1, HRSur2, HRSur3}. The first paper in the series \cite{paperI} dealt with the geometric cycles, the associated six term exact sequence, and the geometrically defined secondary invariants. In the second paper \cite{paperII}, the relationship between these geometrically defined secondary invariants and analytically defined secondary invariants was developed$-$mainly focusing on the relative $\eta$-invariant. In particular, an alternative proof of a theorem of Keswani \cite[Theorem 1.2]{Kes} was developed. The current paper concerns itself with higher invariants: $K$-theoretic invariants giving the isomorphism with Higson-Roe's original analytic model as well as Chern characters in non-commutative de Rham homology. 

Before giving the motivation for this paper and its role in the series of papers, we state the two main theorems of the current paper. 

\begin{theorem*}
For a discrete finitely generated group $\Gamma$ and a compact metric space $X$ with a map $X\to B\Gamma$, there is an isomorphism:
\begin{align*}
\lambda^\mathcal{S}_{an}:\mathcal{S}_*^{geo}(X,\mathcal{L}_X)&\longrightarrow^{\!\!\!\!\!\!\!\!\!\cong}\;\; K_*(D^*(\tilde{X})^\Gamma),\\
\lambda^\mathcal{S}_{an}:(W,\xi,f)&\mapsto j_X\ind_{APS}(W,\Xi,f)-\rho(W,\Xi,f).
\end{align*}
Moreover, $\lambda^\mathcal{S}_{an}$ is compatible with the analytic surgery exact sequence.
\end{theorem*}

Here $K_*(D^*(\tilde{X})^\Gamma)$ is Higson-Roe's analytic surgery group \cite{HRSur1, HRSur2, HRSur3}, explained further in Subsection \ref{subSecCoarseGeo}, and $j_X$ the natural mapping $K_*(C^*_r(\Gamma))\to K_*(D^*(\tilde{X})^\Gamma)$. The triple $(W,\xi,f)$ is a geometric surgery cycle for the Mishchenko bundle $\mathcal{L}_X\to X$ from \cite{paperI} and  $(W,\Xi,f)$ denotes the cycle $(W, \xi, f)$ with further analytic data fixed. A short review of geometric cycles for surgery and the construction of $\lambda$ can be found in Subsection \ref{secttwoone}. The above theorem appears as Theorem \ref{themaplambda} in the main body of the paper; it is the main result of Section \ref{sectioniso}.

\begin{theorem*}
For a discrete finitely generated group $\Gamma$ and a dense Fr\'echet subalgebra $\mathcal{A}\subseteq C^*_r(\Gamma)$, there is a Chern character from the geometric surgery cycles for $\mathcal{A}$ to the delocalized de Rham homology groups of $\mathcal{A}$ 
$$\ch^{del}:\mathcal{S}^{geo}_*(X;\mathcal{A})\to \hat{H}_*^{del,<e>}(\mathcal{A}).$$
The delocalized Chern character and the Chern character of $\mathcal{A}$ fits into a commuting diagram with exact rows 
\tiny
\[
\begin{CD}
\ldots@>>>K^{geo}_{*}(X)@>\mu>> K_*^{geo}(pt;\mathcal{A}) @>r >> \mathcal{S}^{geo}_*(X;\mathcal{A})  @> \delta >> K^{geo}_{*-1}(X)@>>>\ldots  \\
@.@V\ch_X^{<e>}VV@V\ch_{\mathcal{A}(\Gamma)} VV @V\mathrm{ch}^{del}VV @V\ch_{X}^{<e>}VV  \\
\ldots@>>>\hat{H}_{*}^{<e>}(\mathcal{A}(\Gamma))@>\mu_{dR}>> \hat{H}^{dR}_*(\mathcal{A}) @>r_{dR}>> \hat{H}_*^{del,<e>}(\mathcal{A}) @>\delta_{dR}>> \hat{H}_{*-1}^{<e>}(\mathcal{A})@>>>\ldots  \\
\end{CD}.
\]
\end{theorem*}

The geometric model for surgery $\mathcal{S}^{geo}_*(X;\mathcal{A})$ with coefficients in a Fr\'echet  algebra is constructed as in \cite{paperI}. The delocalized de Rham homology is constructed by means of a mapping cone complex, see Subsection \ref{leshom}. The construction of the delocalized Chern character is carried out in Subsection \ref{subsescons}; the previous theorem appears as Theorem \ref{delchethe} in the main body of the paper.\\

As mentioned in the previous papers in this series, we have followed the general approach to index theory put forward by Paul Baum (see for example \cite[Part 5]{BD} or \cite[Introduction]{BE}). We begin with a particular aspect of Baum's approach as motivation for our results. In fact, our results are very much analogous to the situation we discuss below.

Baum's approach begins with the fundamental observation that the $K$-homology of a finite CW-complex (denoted by $X$) admits realizations by two classes of cycles$-$one geometric and the other analytic. On the geometric side, there is the Baum-Douglas model of $K$-homology, while on the analytic side, there is the Kasparov model via Fredholm modules. A cycle in geometric (i.e., Baum-Douglas) $K$-homology is given by $(M,E,f)$ where $M$ is a compact spin$^c$-manifold, $E\to M$ is a vector bundle, and $f$ is a continuous function from $M$ to $X$. The associated $K$-homology group is defined using these cycles and a geometrically defined equivalence relation; we denote it by $K^{geo}_*(X)$. On the other hand, $K^{ana}_*(X)$ is defined to be the Kasparov group $KK_*(C(X), \field{C})$. Further details on the geometric model can be found in \cite{BD,BD2}; for the analytic model the reader is directed to \cite{HR, Kasp}. 

The relationship between these two types of cycles and the associated groups is made mathematically precise through an explicit isomorphism:
$$\lambda: K^{geo}_*(X) \rightarrow K^{ana}_*(X),$$
which is defined at the level of cycles via
$$ (M,E,f) \mapsto f_*([D_E]).$$
where right hand side is the class $[D_E]\in K_*^{ana}(M)$  of the twisted Dirac operator pushed forward by the map (on analytic $K$-homology) induced from $f$. The direction of the map at the level cycles indicates (somewhat informally) that there are ``less" geometric cycles as compared to analytic cycles. As a result, it is often easier to define maps with domain given by the geometric cycles. 

A prototypical example is the construction of the homological Chern character at the level of geometric cycles (see \cite[pages 141-142]{BD}). For a moment suppose that $X$ is a manifold. We let $H_*^{dR}(X)$ denote its compactly supported de Rham homology. Then, the construction of the Chern character is summarized in the following diagram:
\begin{equation}
\label{cherndefka}
 \begin{CD} K^{geo}_*(X) @>\lambda >> K^{ana}_*(X)  \\
@V{\rm ch}VV @. \\ 
H_*^{dR}(X)
\end{CD}
\end{equation}
It is important to note that there does {\it not} exist a map defined at the level of all Fredholm modules which completes this diagram. On the other hand, suppose that we are given a cycle in analytic $K$-homology, but can find a geometric cycle that maps to its class in analytic $K$-homology. Then, the Chern character of the analytic cycle can be {\it defined} using the diagram \eqref{cherndefka}. Examples of this process are discussed in \cite[Part 5]{BD}. There is also a more analytic approach to the Chern character using finitely summable Fredholm modules (see for example \cite{connesncdg,Connesbook}).

Returning to the subject at hand, the two main theorems of the paper can be summarized in a single diagram (which is analogous to the diagram \eqref{cherndefka}):
\begin{center}
$ \begin{CD} \mathcal{S}^{geo}_*(X,\mathcal{L}_X) @>\lambda_{an}^\mathcal{S}>> K_*(D^*(\tilde{X})^\Gamma)  \\
@V\ch^{del}VV @.  \\
\hat{H}_*^{del,<e>}(\mathcal{A}(\Gamma))
\end{CD}
$
\end{center}
The process discussed in the previous paragraph can be applied in this situation whenever the inclusion $\mathcal{A}(\Gamma)\subseteq C^*_r(\Gamma)$ induces an isomorphism on a suitable $K$-theory group, the precise conditions are discussed in Subsection \ref{smootthsubs}. That is, given a class in the analytic surgery group that has an explicit preimage under the isomorphism $\lambda_{an}^\mathcal{S}$, one can {\it define} the delocalized Chern character using the preimage cycle. It is worth noting that at present there is no analytic definition of the delocalized Chern character, see \cite[Section 10]{SchPSCsur}. Of course, a purely analytic approach to the delocalized Chern character is still highly desirable.\\

The paper is structured as follows. We begin with a number of preliminaries$-$most notably concerning coarse geometry and higher Atiyah-Patodi-Singer index theory. The isomorphism from the geometric realization to the original analytic surgery group is discussed in Section \ref{sectioniso} and the definition of delocalized Chern character is given in Section \ref{sectionDeloc}. The general structure of the proofs in these two sections is similar: the heart of the matter is showing the relevant map is well-defined (i.e., respects the equivalence relation used to define the geometric group). In addition, the proofs are less technical in the case of ``easy cycles" (see \cite[Introduction]{paperII}); the reader may wish to consider this special case on a first read. 

The last two sections of the paper treat geometric applications. In Section \ref{mapstotosu} we map Stolz' positive scalar curvature exact sequence into the six-term exact sequence associated to the geometric group; this construction should be compared with the map from Stolz' sequence to the analytic surgery exact sequence. In particular, Theorem \ref{thmDelocPSC} justifies our use of the term ``delocalized Chern character" for the map constructed in Section \ref{sectionDeloc}. We briefly discuss the potential usage of the geometric realization of the analytic surgery group to the surgery exact sequence in topology in Section \ref{surgsurgrem}.

\section{Preliminaries}

We will recall the relevant notions from coarse geometry and higher index theory in this section. When proving the main results of this paper, the techniques of \cite{PSrhoInd,PSsignInd} will play a major role and therefore we mainly use the notation of those papers. We will indicate whenever our notation strays from the notation of \cite{PSrhoInd,PSsignInd}. For an overview of the notions that we will make use of, see \cite[Section 1]{PSrhoInd}. A more detailed exposition can be found in \cite{Roe}. 

The setup that prevails throughout this section, and the two subsequent, is that of a $\Gamma$-presented space, see \cite[Definition 2.6]{HRSur2}. A $\Gamma$-presented space is a proper geodesic metric space $\tilde{X}$ equipped with a free and proper action of $\Gamma$. We set $X=\tilde{X}/\Gamma$.  We will assume that $X$ is compact and often that it is a finite CW-complex. In the notations of \cite{PSrhoInd,PSsignInd}, $X$ denotes the space with a free and proper $\Gamma$-action.

\subsection{Coarse geometry and the analytic surgery exact sequence}
\label{subSecCoarseGeo}

We choose a Hilbert space $\mathpzc{H}_X$ with a $\Gamma$-equivariant representation of $C_0(\tilde{X})$ that is assumed to be non-degenerate and ample. We call such a Hilbert space a $\Gamma$-equivariant $X$-module. The particular choice of $\Gamma$-equivariant $X$-module will only affect the associated invariants by a natural isomorphism (for more details see \cite{Roe}).

Recall the following terminology for operators on a $\Gamma$-equivariant $X$-module $\mathpzc{H}_X$:
\begin{itemize}
\item An operator $T$ is called \emph{locally compact} if $\phi T, T\phi\in \mathcal{K}(\mathpzc{H}_X)$ for any $\phi\in C_0(\tilde{X})$.
\item An operator $T$ is called \emph{pseudo-local} if $[\phi,T]\in \mathcal{K}(\mathpzc{H}_X)$ for any $\phi\in C_0(\tilde{X})$.
\item An operator $T$ is said to have \emph{finite propagation} if for some $R>0$ it holds that $\phi T\phi'=0$ whenever $\phi,\phi'\in C_0(X)$ satisfies $d(\mathrm{supp}(\phi),\mathrm{supp}(\phi'))>R$. Here $d$ denotes the metric on $X$.
\end{itemize}
We note that the set of locally compact operators as well as the set of pseudo-local operators form $C^*$-subalgebras of the bounded operators on $\mathpzc{H}_X$. The $C^*$-algebra of locally compact operators form an ideal in the $C^*$-algebra of pseudo-local operators.

Associated with a $\Gamma$-presented space and the Hilbert space $\mathpzc{H}_X$, there are a number of $C^*$-algebras; we will make use of two: $C^*(\tilde{X},\mathpzc{H}_X)^\Gamma$ and $D^*(\tilde{X},\mathpzc{H}_X)^\Gamma$. The $C^*$-algebra $D^*(\tilde{X},\mathpzc{H}_X)^\Gamma$ is defined to be the $C^*$-algebra closure of all $\Gamma$-invariant pseudo-local operators with finite propagation. The ideal $C^*(\tilde{X},\mathpzc{H}_X)^\Gamma\subseteq D^*(\tilde{X},\mathpzc{H}_X)^\Gamma$ is the intersection of $D^*(\tilde{X},\mathpzc{H}_X)^\Gamma$ with the locally compact operators on $\mathpzc{H}_X$. We recall the following standard facts about the associated $K$-theory groups. A mapping is coarse if for any $R>0$, there is an $S>0$ such that the image of any ball of radius $R$ is contained in a ball of radius $S$.

\begin{theorem}
\label{thenatiso}
The $K$-theory groups of $C^*(\tilde{X},\mathpzc{H}_X)^\Gamma$ and $D^*(\tilde{X},\mathpzc{H}_X)^\Gamma$ have the following properties:
\begin{enumerate}
\item There are natural isomorphisms 
\begin{align*}
\qquad \;\;\quad K_*(C^*(\tilde{X},&\mathpzc{H}_X)^\Gamma)\cong K_*(C^*_r(\Gamma))\quad \mbox{and}\\
&K_*(D^*(\tilde{X},\mathpzc{H}_X)^\Gamma/C^*(\tilde{X},\mathpzc{H}_X)^\Gamma)\cong KK_{*-1}(C(X),\C)=K_{*-1}^{an}(X).
\end{align*}
\item Under these isomorphisms, the boundary mapping 
$$\partial: K_{*+1}(D^*(\tilde{X},\mathpzc{H}_X)^\Gamma/C^*(\tilde{X},\mathpzc{H}_X)^\Gamma)\to K_*(C^*(\tilde{X},\mathpzc{H}_X)^\Gamma)$$
corresponds to the assembly map for free actions $\mu_X:K_*^{an}(X)\to K_*(C^*_r(\Gamma))$.
\item Let $\tilde{Y}$ be another $\Gamma$-represented space and $\mathpzc{H}_Y$ a $\Gamma$-equivariant $Y$-module. Whenever $f:X\to Y$ is a coarse $\Gamma$-equivariant mapping, there is a functorially associated mapping
$$f_*:K_*(C^*(\tilde{X},\mathpzc{H}_X)^\Gamma)\to K_*(C^*(\tilde{Y},\mathpzc{H}_Y)^\Gamma).$$
\item Whenever $f:X\to Y$ is a continuous coarse $\Gamma$-equivariant mapping, there is a functorially associated mapping
$$f_*:K_*(D^*(\tilde{X},\mathpzc{H}_X)^\Gamma)\to K_*(D^*(\tilde{Y},\mathpzc{H}_Y)^\Gamma).$$
\end{enumerate}
\end{theorem}

The reader is directed to \cite{HRAss,HR,Roe} for further details. Motivated by this result, when $\mathpzc{H}_X$ is clear from the context, we simply write  $C^*(\tilde{X})^\Gamma$ and $D^*(\tilde{X})^\Gamma$; up to natural isomorphism, a change of $\mathpzc{H}_X$ does not alter the relevant $K$-groups. 

\begin{remark}
It is possible to construct full versions of $C^*(\tilde{X},\mathpzc{H}_X)^\Gamma$ and $D^*(\tilde{X},\mathpzc{H}_X)^\Gamma$ by using the maximal $C^*$-norm on the $*$-algebra of $\Gamma$-invariant pseudo-local operators with finite propagation. In this case, the boundary map in $K$-theory is the full assembly map for free actions. The constructions of this paper also apply in this framework.
\end{remark}

The next theorem follows from Theorem \ref{thenatiso} and the existence of six term exact sequences in $K$-theory.

\begin{theorem}[The analytic surgery exact sequence, \cite{HRSur1}]
\label{anasurgthm}
There is a six term exact sequence 
\begin{center}
$\begin{CD}
K_0^{an}(X) @>\mu_X>> K_0(C^*_r(\Gamma)) @>j_X>>  K_0(D^*(\tilde{X})^\Gamma)\\
@Aq_XAA @. @VVq_XV \\
 K_1(D^*(\tilde{X})^\Gamma) @<j_X<<  K_1(C^*_r(\Gamma)) @<\mu_X<< K_1^{an}(X)
\end{CD},$
\end{center}
where $j_X:K_*(C^*_r(\Gamma)) \to  K_*(D^*(\tilde{X})^\Gamma)$ is the composition of the isomorphism of Theorem \ref{thenatiso} with the inclusion $C^*(\tilde{X})^\Gamma\to D^*(\tilde{X})^\Gamma$ and $q_X:K_*(D^*(\tilde{X})^\Gamma)\to K_{*+1}^{an}(X)$ the composition of the quotient mapping $D^*(\tilde{X})^\Gamma\to D^*(\tilde{X})^\Gamma/C^*(\tilde{X})^\Gamma$ with the isomorphism of Theorem \ref{thenatiso}.
\end{theorem}

\begin{remark}
The motivation for the terminology ``analytic surgery exact sequence" is twofold. Firstly, it is the analytic analog of the surgery exact sequence and secondly, there is a precise relationship between the two exact sequences (see \cite{HRSur1, HRSur2, HRSur3, PSsignInd} for details); we also discuss this relationship in Section \ref{surgsurgrem}. We will return to the analytic surgery exact sequence in the context of the geometric model from \cite{paperI} in Section \ref{sectioniso}.
\end{remark}

\subsection{Higher Atiyah-Patodi-Singer index theory}

The higher index theory for a manifold with boundary can elegantly be described using the tools of the previous subsection. Parts of the discussion on higher index theory resembles that occurring previously in this series of papers (e.g., \cite[Section 1.2]{paperII}). As such, we keep it short and use the notation of \cite{paperII}. In particular, our notation for Dirac operators will differ from \cite{PSrhoInd,PSsignInd}; we unfortunately need a notation allowing for more indices. They are used to indicate all the possible data associated with a geometric cycles. We will often (for simplicity) restrict the discussion to even-dimensional manifolds with odd-dimensional boundaries, the reader is referred to \cite{LP} for the general situation. 

Let $W$ be a compact spin$^c$ manifold with boundary (potentially with empty boundary) equipped with some metric $g_W$ which is of product type near the boundary. We assume that $D^W$ is a choice of spin$^c$-Dirac operator on $W$ that is of product type near $\partial W$; the operator $D^W$ acts on the sections of the complex spinor bundle $S_W\to W$ associated with the spin$^c$-structure. Since $D^W$ is of product type, there is an associated boundary operator $D^{\partial W}$ -- a spin$^c$-Dirac operator on $\partial W$. 

A smooth locally trivial bundle of finitely generated projective $C^*_r(\Gamma)$-Hilbert modules will be called a $C^*_r(\Gamma)$-bundle. If $\mathcal{E}\to W$ is a $C^*_r(\Gamma)$-bundle equipped with a hermitean connection, there is an associated Dirac type operator $D^W_\mathcal{E}$ acting on $S_W\otimes \mathcal{E}$ that is of product type near $\partial W$. For simplicity, assume that the fibers of $\mathcal{E}$ are full Hilbert $C^*_r(\Gamma)$-modules. As such, there is an associated boundary operator $D^{\partial W}_\mathcal{E}$ acting on $S_{\partial W}\otimes \mathcal{E}|_{\partial W}$, coinciding with the Dirac type operator constructed from $\mathcal{E}|_{\partial W}$ and $D^{\partial W}$. If $W$ is even-dimensional, $D^W$ is odd in the grading on the graded bundle $S_W$. If $W$ is odd-dimensional, $\partial W$ is even-dimensional and $D^{\partial W}$ is odd in the grading on the graded bundle $S_{\partial W}$.

If $W$ has a boundary $\partial W$, we let $W_\infty$ denote the complete manifold obtained from gluing on the cylinder $[0,\infty)\times \partial W$, equipped with the metric $\mathrm{d} t^2+g_W|_{\partial W}$, on to $W$. Any $C^*_r(\Gamma)$-bundle $\mathcal{E}\to W$ extends to $W_\infty$ and the same holds for Dirac operators $D^W_\mathcal{E}$. We denote these extensions by $\mathcal{E}_\infty\to W_\infty$ and respectively $D^{W_\infty}_\mathcal{E}$.  

If $W=M$ is a closed manifold, the operator $D^M_\mathcal{E}$ is an elliptic element in the Mishchenko-Fomenko calculus $\Psi^*_{C^*_r(\Gamma)}(M,\mathcal{E})$. Hence it defines regular, self-adjoint, $C^*_r(\Gamma)$-Fredholm operator on the $C^*_r(\Gamma)$-Hilbert module $L^2(M,\mathcal{E})$. In a larger generality, a Dirac operator on a $C^*$-bundle over a complete manifold defined from a hermitean connection is always self-adjoint and regular, see \cite[Theorem 2.3]{hankpapsch}. Returning to closed manifolds, the $C^*_r(\Gamma)$-Fredholm operator $D^M_\mathcal{E}$ has a higher index
$$\ind_{C^*_r(\Gamma)}(D^M_\mathcal{E})\in K_{\dim (M)}(C^*_r(\Gamma)).$$
The parity of the $K$-theory group comes from the dimension of $M$: if $\dim(M)$ is even, $D^M_\mathcal{E}$ is self-adjoint and graded while if $\dim(M)$ is odd, $D^M_\mathcal{E}$ is merely self-adjoint. Recall that by \cite[Theorem 3]{LP} and \cite[Proposition 2.10]{LPGAFA}, there is a smoothing operator $A\in \Psi^{-\infty}_{C^*_r(\Gamma)}(M,\mathcal{E})$ such that $D^M_\mathcal{E}+A$ is invertible if and only if $\ind_{C^*_r(\Gamma)}(D^M_\mathcal{E})=0$. Such an operator is called a \emph{trivializing operator} for $D^M_\mathcal{E}$. For even-dimensional $M$, a trivializing operator $A$ is by definition odd. The projection 
$$P_A:=\frac{1}{2}\left(\frac{D^M_\mathcal{E}+A}{|D^M_\mathcal{E}+A|}+1\right),$$
is called a \emph{spectral section} for $D^M_\mathcal{E}$, see \cite[Definition 3]{LP} and \cite[Proposition 2.10]{LPGAFA}.

In the particular situation when $M=\partial W$ and $D^M_\mathcal{E}=D^{\partial W}_\mathcal{E}$, it follows from \cite[Theorem 6.2]{hilsumbordism} that $\ind_{C^*_r(\Gamma)}(D^M_\mathcal{E})=0$. For any such choice of smoothing operator $A\in \Psi^{-\infty}_{C^*_r(\Gamma)}(M,\mathcal{E})$, we let $A_\infty$ denote the odd operator on $L^2(W_\infty, S_{W_\infty}\otimes \mathcal{E}_\infty)$ constructed from the operator 
$$\chi\otimes A:L^2([0,\infty)\times \partial W, S_{[0,\infty)\times \partial W}^+\otimes \mathcal{E}_\infty|_{[0,\infty)\times \partial W})\circlearrowleft,$$
for some $\chi\in C^\infty[0,\infty)$ vanishing near $0$ and equaling $1$ near $\infty$. The proof of the next theorem follows (more or less) by combining these facts with the techniques of \cite[Section 6]{LPGAFA}.

\begin{theorem}
\label{howtodefaps}
Let $W$ be a compact even-dimensional spin$^c$ manifold with boundary equipped with the following data:
\begin{itemize}
\item a $C^*_r(\Gamma)$-bundle $\mathcal{E}\to W$;
\item a spin$^c$-Dirac operator $D^W$ on $W$;
\item a self-adjoint smoothing operator $A\in \Psi^{-\infty}_{C^*_r(\Gamma)}(\partial W,\mathcal{E}|_{\partial W})$ such that $D^{\partial W}_\mathcal{E}+A$ is invertible.
\end{itemize}
The associated $C^*_r(\Gamma)$-linear operator $D^{W_\infty}_\mathcal{E}+A_\infty$ is a regular, self-adjoint, $C^*_r(\Gamma)$-Fredholm operator on $L^2(W_\infty, S_{W_\infty}\otimes \mathcal{E}_\infty)$.
\end{theorem}

\begin{define}
Let $\ind_{APS}(D^W_\mathcal{E},A):=\ind_{C^*_r(\Gamma)}(D^{W_\infty}_\mathcal{E}+A_\infty)\in K_{\dim W}(C^*_r(\Gamma))$.
\end{define}

\begin{remark}
The constructions in \cite[Section 2.3]{PSsignInd} of APS-indices are made in the context of coarse geometry and differs slightly from approach above. The geometric setup considered in \cite{PSsignInd} is as follows.  One has a continuous mapping $g:W\to B\Gamma$, a vector bundle $E\to W$ and $\mathcal{E}=E\otimes g^*\mathcal{L}_{B\Gamma}$, where $\mathcal{L}_{B\Gamma}:=E\Gamma\times_\Gamma C^*_r(\Gamma)\to B\Gamma$ denotes the Mishchenko bundle. For this choice of $C^*_r(\Gamma)$-bundle there is a geometrically constructed isomorphism $L^2(\tilde{W},S_{\tilde{W}}\otimes \tilde{E})\cong L^2(W,S_W\otimes \mathcal{E})\otimes_{\pi} \ell^2(\Gamma)$, where $\pi:C^*_r(\Gamma)\to \mathcal{B}(\ell^2(\Gamma))$ denotes the regular representation. 

A Dirac operator $D_E^W$ on $E\to W$ can be lifted to a $\Gamma$-equivariant operator $\tilde{D}_E^W$ on the Galois covering $\tilde{W}:=E\Gamma\times_g W\to W$. Under the isomorphism above, $D^W_\mathcal{E}\otimes_\pi\mathrm{id}_{\ell^2(\Gamma)}$ corresponds to $\tilde{D}_E^W$ and $A\otimes \mathrm{id}_{\ell^2(\Gamma)}$ to an operator $\tilde{A}$ on $L^2(\partial\tilde{W}, S_{\partial\tilde{W}}\otimes \tilde{E}|_{\partial \tilde{W}})$ such that $\tilde{D}_E^{\partial W}+\tilde{A}$ is $L^2$-invertible. In \cite[Section 2.3]{PSsignInd}, there is a construction of the index class $\ind_{C^*(\tilde{W})^\Gamma}(\tilde{D}_E^{W}+\tilde{A}_\infty)\in K_{\dim W}(C^*(\tilde{W})^\Gamma)$; it is the image of $\ind_{APS}(D^W_\mathcal{E},A)$ under the isomorphism $K_*(C^*(\tilde{W})^\Gamma)\cong K_*(C^*_r(\Gamma))$.
\end{remark}

\begin{theorem}[Gluing higher APS-indices]
\label{gluingthm}
Let $Z$ be a compact even-dimensional spin$^c$-manifold and $Y\subseteq Z^\circ$ a hyper surface separating $Z$ into two compact manifolds, $W_1$ and $W_2$. That is, $Z=W_1\cup_YW_2$ with $\partial W_i=M_i\dot{\cup} (-1)^iY$ for $i=1,2$ and closed manifolds $M_1$ and $M_2$. Furthermore, assume that the following data has been fixed:
\begin{itemize}
\item a spin$^c$-Dirac operator $D^Z$ on $Z$ of product type near $\partial Z$ and $Y$, with associated boundary operators $D^{M_1}$, $D^{M_2}$ respectively $D^Y$;
\item a $C^*_r(\Gamma)$-bundle $\mathcal{E}\to Z$ with connection of product type near $\partial Z$ and $Y$;
\item trivializing operators $A_i\in \Psi^{-\infty}_{C^*_r(\Gamma)}(M_i, S_{M_i}\otimes \mathcal{E}|_{M_i})$ of $D^{M_i}_{\mathcal{E}|_{M_i}}$, for $i=1,2$;
\item a trivializing operator $A_Y\in \Psi^{-\infty}_{C^*_r(\Gamma)}(Y, S_{Y}\otimes \mathcal{E}|_Y)$ of $D^{Y}_{\mathcal{E}|_Y}$.
\end{itemize}
Then,
$$\ind_{APS}(D^Z_\mathcal{E},A_1\dot{\cup}A_2)=\ind_{APS}(D^{W_1}_{\mathcal{E}|_{W_1}},A_1\dot{\cup}-A_Y)+\ind_{APS}(D^{W_2}_{\mathcal{E}|_{W_2}},A_2\dot{\cup}A_Y),$$
where $D^{W_i}$ is the restriction of $D^Z$ to $W_i$, for $i=1,2$.
\end{theorem}

Theorem \ref{gluingthm} follows from applying \cite[Theorem 8]{LP} to the hyper surface $\partial W_1\cup \partial W_2$ in the closed manifold $2Z$. 

\begin{remark}
The assumption on existence of the trivializing operators $A_1$, $A_2$ and $A_Y$ in Theorem \ref{gluingthm} is necessary and does not follow from \cite[Theorem 3]{LP}. In general, \cite[Theorem 3]{LP} only implies the existence of a trivializing operator on $\partial W_1$ and $\partial W_2$, but not that it decomposes into a direct sum acting on the disjoint union $\partial W_i=M_i\dot{\cup} (-1)^iY$ for $i=1,2$.
\end{remark}

\begin{define}
\label{sfdef}
Let $\mathcal{E}\to M$ be a $C^*_r(\Gamma)$-bundle on a closed odd-dimensional manifold $M$, $D^M_{\mathcal{E},\bullet}=(D^M_{\mathcal{E},u})_{u\in [0,1]}$ be a family of Dirac operators with vanishing index, and $P_0$ and $P_1$ are spectral sections for $D^M_{\mathcal{E},0}$ and respectively $D^M_{\mathcal{E},1}$. Then, the spectral flow of $(D^M_{\mathcal{E},\bullet},P_0,P_1)$ is defined as
$$\mathrm{sf}(D^M_{\mathcal{E},\bullet},P_0,P_1):=[P_1-Q_1]-[P_0-Q_0]\in K_0(C^*_r(\Gamma)),$$
where $Q_{\bullet}=(Q_t)_{t\in [0,1]}$ is spectral section for $D^M_{\mathcal{E},\bullet}$ viewed as a $C[0,1]\otimes C^*_r(\Gamma)$-linear operator. 
\end{define}

For more details regarding Definition \ref{sfdef}, see \cite[Section 2.5]{LP}. We state \cite[Proposition 8]{LP} for the convenience of the reader:

\begin{theorem}[Change of data in higher APS-indices]
\label{datachangingaps}
Let $W$ be a compact even-dimensional spin$^c$-manifold and $\mathcal{E}\to W$ a $C^*_r(\Gamma)$-bundle. Assume that $D^W_\mathcal{E}$ and $D'^W_\mathcal{E}$ are two Dirac type operators on $S_W\otimes \mathcal{E}$ as above. Take two trivializing operators $A,A'\in \Psi^{-\infty}_{C^*_r(\Gamma)}(\partial W,S_{\partial W}\otimes \mathcal{E}|_{\partial W})$ for $D^{\partial W}_\mathcal{E}$ respectively $D'^{\partial W}_\mathcal{E}$. Let $P:=\chi_{[0,\infty)}(D^{\partial W}_\mathcal{E}+A)$ and $P':=\chi_{[0,\infty)}(D'^{\partial W}_\mathcal{E}+A')$. Then,
$$\ind_{APS}(D^W_\mathcal{E},A)-\ind_{APS}(D'^W_\mathcal{E},A')=\mathrm{sf}((1-t)D^{\partial W}_\mathcal{E}+tD'^{\partial W}_\mathcal{E},P,P').$$
\end{theorem}

\begin{remark}
\label{differenceofpr}
In \cite{LP}, a more convenient notation was introduced to treat the case $D^{\partial W}_\mathcal{E}=D'^{\partial W}_\mathcal{E}$. For a $C^*_r(\Gamma)$-bundle $\mathcal{E}\to M$ on a closed odd-dimensional manifold $M$ and any two spectral sections $P$ and $P'$ of a Dirac operator $D^M_\mathcal{E}$, there exists a $C^*_r(\Gamma)$-linear projection 
$$Q\in \mathcal{L}_{C^*_r(\Gamma)}(L^2(M,S_{M}\otimes \mathcal{E})),$$
such that $P-Q$ and $P'-Q$ are projections in $\mathcal{K}_{C^*_r(\Gamma)}(L^2(M,S_{M}\otimes \mathcal{E})$ and one sets 
$$[P-P']:=[P-Q]-[P'-Q]\in K_0(C^*_r(\Gamma)),$$
using the isomorphism $K_0(\mathcal{K}_{C^*_r(\Gamma)}(L^2(M,S_{M}\otimes \mathcal{E})))\cong K_0(C^*_r(\Gamma))$ induced by Morita equivalence (for further details see \cite[Corollary 1]{LP}).
\end{remark}

\subsection{Higher $\rho$-invariants}

Assume that $M$ is a closed odd-dimensional spin$^c$-manifold equipped with a choice of spin$^c$-Dirac operator $D^M$ defining the fundamental class $[M]:=[D^M]\in K_1^{an}(M)$. Also, assume that $f:M\to B\Gamma$ is a continuous map; in particular, $\tilde{M}:= E\Gamma\times_f M$ is a $\Gamma$-presented space with $\tilde{M}/\Gamma=M$. Associated with this data, there is a Mishchenko bundle:
$$\mathcal{L}_M:=\tilde{M}\times_{\Gamma} C^*_r(\Gamma)\cong f^*\mathcal{L}_{B\Gamma}.$$
The Mishchenko bundle $\mathcal{L}_M$ is a flat $C^*_r(\Gamma)$-bundle. It is not difficult to see that for a vector bundle $E\to M$ and a spin$^c$-Dirac operator $D^M_E$,
$$\mu_M([M]\cap [E])=\ind_{C^*_r(\Gamma)}(D^M_{E\otimes \mathcal{L}_M}).$$
Here $D^M_{E\otimes \mathcal{L}_M}$ denotes the Dirac operator on $S_M\otimes E\otimes \mathcal{L}_M$ obtained from twisting $D^M_E$ by the flat connection on $\mathcal{L}_M$. 

If $\ind_{C^*_r(\Gamma)}(D^M_{E\otimes \mathcal{L}_M})=0$, there is a particular choice of pre image in $K_0(D^*(\tilde{M})^\Gamma)$ associated with a choice of trivializing operator $A\in \Psi^{-\infty}_{C^*_r(\Gamma)}(M,S_M\otimes E\otimes \mathcal{L}_M)$. Recall that for a smoothing operator $A\in \Psi^{-\infty}_{C^*_r(\Gamma)}(M,S_M\otimes E\otimes \mathcal{L}_M)$, $\tilde{A}$ denotes the smoothing operator on $L^2(\tilde{M},S_{\tilde{M}}\otimes \tilde{E})$ obtained from $A\otimes_{\pi}\mathrm{id}_{\ell^2(\Gamma)}$. The following theorem is a consequence of the discussion in \cite[Section 2.2]{PSsignInd}; the key technical result is \cite[Proposition 2.8]{PSsignInd}. 

\begin{theorem}
\label{rhodefandcomq}
Let $M$, $f$, $E$ and $D^M_E$ be as in the paragraphs just before this statement. Then, for any trivializing operator $A\in \Psi^{-\infty}_{C^*_r(\Gamma)}(M,S_M\otimes E\otimes \mathcal{L}_M)$, the element $\chi_{[0,\infty)}(\tilde{D}^M_{E}+\tilde{A})$ is a well defined projection in $D^*(\tilde{M})^\Gamma$ giving the \emph{$\rho$-class}:
$$\rho(D^M_{E},A):=\left[\chi_{[0,\infty)}(\tilde{D}^M_{E}+\tilde{A})\right]\in K_0(D^*(\tilde{M})^\Gamma).$$
Moreover, $\rho(D^M_{E},A)$  lifts $[M]\cap [E]$, in the sense that 
$$q_X\left(\rho(D^M_{E},A)\right)=[M]\cap [E]\quad\mbox{in}\quad K_1^{an}(M).$$
Here the map $q_X$ is as in Theorem \ref{anasurgthm}.
\end{theorem}

The following lemma will come in handy when comparing higher $\rho$-invariants constructed from different geometric data. As in subsection \ref{subSecCoarseGeo}, $\tilde{M}$ denotes a $\Gamma$-presented space with $M=\tilde{M}/\Gamma$. Moreover, we let $j_M:K_*(C^*_r(\Gamma))\to K_*(D^*(\tilde{M})^\Gamma)$ denote the mapping induced from the isomorphism $K_*(C^*_r(\Gamma))\cong K_*(C^*(\tilde{M})^\Gamma)$ and the inclusion $C^*(\tilde{M})^\Gamma\to D^*(\tilde{M})^\Gamma$.

\begin{lemma}
\label{sfandrho}
Suppose that $f:M\to B\Gamma$ is a continuous mapping from a closed manifold and that $E\to M$ is a vector bundle. Let $(D_{E,u}^M)_{u\in [0,1]}$ be a family of Dirac operators thereon and $A_0$ and let $A_1$ be trivializing operators for the Dirac operators $D_{E\otimes \mathcal{L},0}^M$ respectively $D_{E\otimes \mathcal{L},1}^M$. We let $P_0:=\chi_{[0,\infty)}(D^{M}_{E\otimes \mathcal{L},0}+A_0)$ and $P_1:=\chi_{[0,\infty)}(D^{M}_{E\otimes \mathcal{L},1}+A_1)$ denote the associated spectral sections. Then 
$$j_M\left[\mathrm{sf}(D_{E,\bullet}^M, P_0,P_1)\right]=\rho(D^M_{E,1},A_1)-\rho(D^M_{E,0},A_0)\in K_0(D^*(\tilde{M})^\Gamma).$$
\end{lemma}

\begin{proof}
We set $\tilde{P}_i:=\chi_{[0,\infty)}(\tilde{D}^M_{E,i}+\tilde{A}_i)$, for $i=0,1$, so $\rho(D^M_{E,i},A_i)=[\tilde{P}_i]$. It follows from \cite[Proposition 2.8]{PSsignInd} and \cite[Proposition 2.10]{LPGAFA} that we can find a path $(Q_t)_{t\in [0,1]}$ of spectral sections for $(D_{E,u}^M)_{u\in [0,1]}$ with the following property: for $i=0,1$, the projections $\tilde{Q}_i\in D^*(\tilde{M})^\Gamma$ have the property that $\tilde{P}_0-\tilde{Q}_0, \tilde{P}_1-\tilde{Q}_1\in C^*(\tilde{M})^\Gamma$ are projections and the class of $\mathrm{sf}(D_{E,\bullet}^M, P_0,P_1)$ is under $K_0(C^*_r(\Gamma))\cong K_0(C^*(\tilde{M})^\Gamma)$  equal to $[\tilde{P}_1-\tilde{Q}_1]-[\tilde{P}_0-\tilde{Q}_0]$ (see Definition \ref{sfdef} and Remark \ref{differenceofpr}), where (for $i=0,1$) the projections $\tilde{Q}_i$ are defined by taking the image of $Q_i\otimes_\pi\mathrm{id}_{\ell^2(\Gamma)}$ under the isomorphism 
$$\mathcal{B}(L^2(M,S_M\otimes E\otimes \mathcal{L}_M)\otimes_\pi \ell^2(\Gamma))\cong  \mathcal{B}(L^2(\tilde{M},S_{\tilde{M}}\otimes \tilde{E})).$$ 
The proposition now follows from the standard fact in $K$-theory that if $p$ and $q$ are projections in a ring $R$ with $pq=0$, then $[p]+[q]=[p+q]$ in $K_0(R)$ (see for example \cite[Lemma 3.4]{hilsumindexclass}).
\end{proof}

After suitable modifications to the definition of the $\rho$-invariant, both Theorem \ref{rhodefandcomq} and Lemma \ref{sfandrho} apply in the even-dimensional case (for more details see \cite{PSrhoInd} for instance).

\begin{theorem}[Delocalized APS-theorem]
\label{delapsthemps}
Let $W$ be a compact manifold, $F\to W$ a vector bundle, $D^W_F$ a Dirac operator on $F$ and $f:W\to B\Gamma$ a continuous mapping with associated Mishchenko bundle $\mathcal{L}_W$. For any trivializing operator $A$ of $D^{\partial W}_{F|_{\partial W}\otimes \mathcal{L}_{\partial W}}$, 
$$j_{W}\left(\ind_{APS}(D^W_{F\otimes \mathcal{L}_W},A)\right)=\iota_*\rho(D^{\partial W}_{F|_{\partial W}},A)\quad\mbox{in}\quad K_*(D^*(\tilde{W})^\Gamma),$$
where $\iota:\partial W\to W$ is the continuous inclusion and $\iota_*:K_*(D^*(\partial \tilde{W})^\Gamma)\to K_*(D^*(\tilde{W})^\Gamma)$ denotes the functorially associated mapping.
\end{theorem}

The delocalized APS-theorem is stated as \cite[Theorem 3.1]{PSsignInd} and proved there in the even-dimensional case. Other proofs of the delocalized APS-theorem (in the less general positive scalar curvature case) can be found in \cite{PSrhoInd,xieyu}.

\begin{remark}
\label{factoringremark}
Suppose $g_X:X\to B\Gamma$ is continuous and $\tilde{X}\cong E\Gamma\times_{g_X}X$ is a $\Gamma$-represented space with $X$ compact. If $f|_{\partial W}$ equals the restriction to the boundary of the composition of continuous maps $\hat{f}:W\to X$ and $g_X:X\to B\Gamma$, then 
\begin{equation}
\label{factoringrho}
j_{X}\left(\ind_{APS}(D^W_{F\otimes \mathcal{L}_W},A)\right)=(\hat{f}|_{\partial W})_*\left(\rho(D^{\partial W}_{F|_{\partial W}},A)\right),
\end{equation}
as elements in $K_*(D^*(\tilde{X})^\Gamma)$. The consequence \eqref{factoringrho} of the delocalized APS-theorem follows because $j_X:K_*(C^*_r(\Gamma))\to K_*(D^*(\tilde{X})^\Gamma)$ factors as $j_X=\hat{f}_*\circ j_W$ and $(\hat{f}|_{\partial W})_*:K_*(D^*(\partial \tilde{W})^\Gamma)\to K_*(D^*(\tilde{X})^\Gamma)$ factors as $(\hat{f}|_{\partial W})_*=\hat{f}_*\circ \iota_*$. This fact will play a major role in proving that the isomorphism between the geometric model and the analytic surgery group respects the bordism relation.
\end{remark}

\section{The isomorphism $\lambda_{an}^\mathcal{S}:\mathcal{S}^{geo}_*(X,\mathcal{L}_X)\xrightarrow{\sim} K_*(D^*(\tilde{X})^\Gamma)$}
\label{sectioniso}

Our goal is the construction of an isomorphism $\lambda_{an}^\mathcal{S}:\mathcal{S}^{geo}_*(X,\mathcal{L}_X)\to K_*(D^*(\tilde{X})^\Gamma)$, which is compatible with the analytic surgery exact sequence of Theorem \ref{anasurgthm}. The map is defined using higher APS-index theory and higher $\rho$-invariants. The isomorphism $\lambda_{an}^\mathcal{S}$ (in a sense) measures the failure of the delocalized APS-index formula to hold when the interior bundle on the manifold with boundary is a general $C^*_r(\Gamma)$-bundle; this will be made more precise below in Remark \ref{delapsandszero}.

\subsection{Decorated surgery cycles and associated higher invariants}
\label{secttwoone}

In order to define the mapping $\lambda_{an}^\mathcal{S}$, we need to recall some notions from the first two paper in this series \cite{paperI,paperII}. Similar to above, we assume that $X$ is a compact metric space and $g_X:X\to B\Gamma$ is a fixed continuous mapping. The space $\tilde{X}:=E\Gamma\times_{g_X}X$ can naturally be equipped with the structure of a $\Gamma$-presented space. The Mishchenko bundle $\mathcal{L}_{B\Gamma}:=E\Gamma\times_\Gamma C^*_r(\Gamma)\to B\Gamma$ is a locally trivial bundle of finitely generated projective $C^*_r(\Gamma)$-modules and 
$$\mathcal{L}_X:=\tilde{X}\times_\Gamma C^*_r(\Gamma)=g_X^*\mathcal{L}_{B\Gamma}.$$ 
The Mishchenko bundle can be defined with respect to other completions as well. In fact, the geometric description of the analytic surgery group holds in an even larger generality, see \cite{paperI}. 

Whenever $(Z,Q,g)$ is a triple consisting of a compact Hausdorff space $Z$, a closed subset $Q\subseteq Z$ and a continuous mapping $f:Q\to X$, we can define relative $K$-theory cocycles for assembly on $(Z,Q,f)$ along $\mathcal{L}_X$ to be quadruples $\xi=(\mathcal{E}_{C^*_r(\Gamma)},\mathcal{E}_{C^*_r(\Gamma)}',E_\C,E_\C',\alpha)$ consisting of 
\begin{itemize}
\item a pair of locally trivial bundles of finitely generated projective $C^*_r(\Gamma)$-modules $\mathcal{E}_{C^*_r(\Gamma)},\mathcal{E}_{C^*_r(\Gamma)}'\to Z$;
\item a pair of vector bundles $E_\C,E_\C'\to Q$;
\item an isomorphism $\alpha:\mathcal{E}_{C^*_r(\Gamma)}|_Q\oplus (E_\C'\otimes f^*\mathcal{L}_X)\to \mathcal{E}_{C^*_r(\Gamma)}'|_Q\oplus (E_\C\otimes f^*\mathcal{L}_X)$ of $C^*_r(\Gamma)$-bundles.
\end{itemize}
This definition is \cite[Definition 1.8]{paperI}. Equipped with a suitable equivalence relation, the relative $K$-theory cocycles for assembly on $(Z,Q,f)$ along $\mathcal{L}_X$ can be used to build a $K$-theory group fitting together with assembly on the level of $K$-theory. For details, see \cite[Section 1.3]{paperI}. We will tacitly assume that the fibers of $\mathcal{E}_{C^*_r(\Gamma)},\mathcal{E}_{C^*_r(\Gamma)}'\to Z$ are full Hilbert $C^*_r(\Gamma)$-modules. Full fibers can always be obtained without changing the $K$-theory class after adding a summand of the trivial $C^*_r(\Gamma)$-bundle to $\mathcal{E}_{C^*_r(\Gamma)}$ and $\mathcal{E}_{C^*_r(\Gamma)}'$.

Recall from \cite[Definition 2.1]{paperI} that a cycle for $\mathcal{S}^{geo}_0(X,\mathcal{L}_X)$ is a triple $(W,\xi,f)$ where:
\begin{itemize}
\item $W$ is an even-dimensional compact spin$^c$-manifold with boundary;
\item $f:\partial W\to X$ is a continuous mapping;
\item $\xi$ is a relative $K$-theory cocycle for assembly on $(W,\partial W,f)$ along $\mathcal{L}_X$.
\end{itemize}
From \cite[Definition 3.1]{paperII} we recall that a quintuple $(D_{\mathcal{E}},D_{\mathcal{E}'},D_E,D_{E'})$ is called a choice of Dirac operators on $\xi$ if 
\begin{itemize}
\item $D_{\mathcal{E}}$ and $D_{\mathcal{E}'}$ are Dirac operators on $S_W\otimes \mathcal{E}_{C^*_r(\Gamma)}$ respectively $S_W\otimes \mathcal{E}^{\prime}_{C^*_r(\Gamma)}$ defined from $C^*_r(\Gamma)$-Clifford connections of product type near $\partial W$. 
\item $D_{E}$ and $D_{E'}$ are Dirac operators on $S_{\partial W}\otimes E_{\field{C}}$ respectively $S_{\partial W}\otimes E_{\field{C}}^{\prime}$. 
\end{itemize}
Since the Dirac operators $D_{\mathcal{E}}$ and $D_{\mathcal{E}'}$ are defined from connections of product type near the boundary, there are associated boundary Dirac operators $D_{\mathcal{E}}^{\partial W}$ and $D_{\mathcal{E}'}^{\partial W}$ acting on $\mathcal{E}\otimes S_{\partial W}$ respectively $\mathcal{E}'\otimes S_{\partial W}$.

\begin{define}[Definition 3.13 of \cite{paperII}]
Let $(W,\xi,f)$ be a cycle for $\mathcal{S}^{geo}_*(X,\mathcal{L}_X)$. A decoration of $\xi$ is a collection $\Xi=(\xi,(D_{\mathcal{E}},D_{\mathcal{E}'},D_E,D_{E'}),(A^\mathcal{E},A^{\mathcal{E}'},A))$ where $(D_{\mathcal{E}},D_{\mathcal{E}'},D_E,D_{E'})$ is a choice of Dirac operators and $(A^\mathcal{E},A^{\mathcal{E}'},A)$ is a choice of trivializing operators for $(W,\xi,f)$, i.e., smoothing operators in the Mishchenko-Fomenko calculus of the following type:
\begin{itemize}
\item $A^{\mathcal{E}}\in \Psi^{-\infty}_{C^*_r(\Gamma)}(\partial W,\mathcal{E}\otimes S_{\partial W})$ is a trivializing operator for $D^{\partial W}_{\mathcal{E}}$;
\item $A^{\mathcal{E}'}\in \Psi^{-\infty}_{C^*_r(\Gamma)}(\partial W,\mathcal{E}'\otimes S_{\partial W})$ is a trivializing operator for $D^{\partial W}_{\mathcal{E}'}$;
\item $A\in \Psi^{-\infty}_{C^*_r(\Gamma)}(\partial W\dot{\cup}-\partial W,(E'\otimes S_{\partial W}\otimes f^*\mathcal{L})\,\dot{\cup}-(E\otimes S_{\partial W}\otimes f^*\mathcal{L}))$ is a trivializing operator for $\left(D^{\partial W}_{E'\otimes f^*\mathcal{L}}\dot{\cup}-D^{\partial W}_{E\otimes f^*\mathcal{L}}\right)$.
\end{itemize}

A decoration of a cycle $(W,\xi,f)$ is a triple $(W,\Xi,f)$ where $\Xi$ is a decoration of $\xi$.
\end{define}

It was shown in \cite[Lemma 3.14]{paperII} that any cycle admits a decoration; the proof is based on \cite[Theorem 3]{LP} and \cite[Proposition 10]{LPGAFA}. 

In order to construct the higher APS-index class and relevant $\rho$-invariants, we need to introduce further notation. Let $(W,\Xi,f)$ be a decorated cycle for $\mathcal{S}^{geo}_*(X,\mathcal{L}_X)$. Associated with the decoration $\Xi$, we construct the Dirac operator $\bar{D}_\Xi^{\partial W\times [0,1]}$, acting on the bundle $$S_{\partial W\times [0,1]}\otimes (\mathcal{E}_{C^*_r(\Gamma)}|_{\partial W} \oplus E_\C'\otimes f^*\mathcal{L}_X)$$ over the cylinder $\partial W\times [0,1]$, by means of specifying the boundary operator $D_\mathcal{E}^{\partial W}\oplus -D_{E'\otimes f^*\mathcal{L}_X}$ on $\partial W\times \{0\}$ and the boundary operator $\alpha^*(-D_{\mathcal{E}'}^{\partial W}\oplus D_{E\otimes f^*\mathcal{L}_X})$ on $\partial W\times \{1\}$. 

We also construct the smoothing operator $A^\Xi$ on the bundle 
$$S_{\partial W\times \{0,1\}}\otimes (\mathcal{E}_{C^*_r(\Gamma)}|_{\partial W} \oplus E_\C'\otimes f^*\mathcal{L}_X)$$
by taking 
$$A^\Xi:=\mathrm{id}\dot{\cup}\alpha^*\left(\left(A^\mathcal{E}\dot{\cup} (-A^{\mathcal{E}'})\right)\oplus A\right).$$ 
By construction, the boundary operator $\bar{D}_\Xi^{\partial W\times \{0,1\}}$ of $D_\Xi^{\partial W\times [0,1]}$ has the property that $\bar{D}_\Xi^{\partial W\times \{0,1\}}+A^\Xi$ is invertible.

\begin{define}
\label{defininginvaria}
For a decorated cycle $(W,\Xi,f)$ for $\mathcal{S}^{geo}_*(X,\mathcal{L}_X)$, we define the classes 
\begin{align*}
\ind_{APS}(W,\Xi,f)&:=\ind_{APS}(D_\mathcal{E}^W,A^\mathcal{E})\\
&\qquad+\ind_{APS}(D_{\mathcal{E}'}^{-W},-A^{\mathcal{E}'})-\ind_{APS}(\bar{D}_\Xi,A^\Xi)\in K_*(C^*_r(\Gamma));\\
\rho(W,\Xi,f)&:=f_*\left[\rho(D_{E}\dot{\cup}-D_{E'},A)\right]\in K_*(D^*(\tilde{X})^\Gamma);\\
\lambda^\mathcal{S}_{an}(W,\Xi,f)&:=j_X\ind_{APS}(W,\Xi,f)-\rho(W,\Xi,f)\in K_*(D^*(\tilde{X})^\Gamma).
\end{align*}
\end{define}

The construction of the APS-index of a decorated cycle can be found in \cite[Definition 3.15]{paperII}.

\begin{remark}
Somewhat informally, the APS-index of a decorated cycle is to be thought of as the sum of the APS-index of $(W,\mathcal{E}_{C^*_r(\Gamma)})\dot{\cup}(-W,\mathcal{E}'_{C^*_r(\Gamma)})$  ``glued in the $K$-theory sense" along the cylinder $\partial W\times [0,1]$ with the bundle data $\mathcal{E}_{C^*_r(\Gamma)}\oplus E'_\C\otimes \mathcal{L}_X$. Still informally, the higher $\rho$-class of a decorated cycle should be thought of as a difference of the higher $\rho$-class of $E_\C$ with that of $E'_\C$. In general, however, the higher $\rho$-class of $E_\C$ and $E'_\C$ are not separately well-defined!

The informal discussion in the previous paragraph is motivated by the following special case: suppose that $\mathcal{E}'_{C^*_r(\Gamma)}=E'_\C=0$ and we take the decoration 
$$\Xi=((\mathcal{E}_{C^*_r(\Gamma)},E_\C,\alpha),(D_\mathcal{E},D_E),(A^\mathcal{E},A)).$$ 
Then, Theorem \ref{gluingthm} implies that $\ind_{APS}(W,\Xi,f)=\ind_{APS}(\hat{D}_\mathcal{E},A)$, where $\hat{D}_\mathcal{E}$ coincides with $D_\mathcal{E}$ outside a neighborhood of $\partial W$ where $\hat{D}_\mathcal{E}$ is of product type with boundary operator $\alpha^*(D_{E\otimes \mathcal{L}})$. In particular,
$$\lambda^\mathcal{S}_{an}(W,\Xi,f)=j_X[\ind_{APS}(\hat{D}_\mathcal{E},A)]-f_*[\rho(D_{E},A)]\in K_*(D^*(\tilde{X})^\Gamma).$$
\end{remark}

\begin{lemma}
\label{depeondec}
Given a cycle $(W,\xi,f)$ for $\mathcal{S}_*(X,\mathcal{L}_X)$, the class $\lambda^\mathcal{S}_{an}(W,\Xi,f)\in K_*(D^*(\tilde{X})^\Gamma)$ does not depend on the choice of decoration $(W,\Xi,f)$ of $(W,\xi,f)$.
\end{lemma}

\begin{proof}
For $i=0$ and $1$, let $\Xi_i=(\xi,(D_{\mathcal{E}_i},D_{\mathcal{E}',i},D_{E,i},D_{E',i}),(A^\mathcal{E}_i,A^{\mathcal{E}'}_i,A_i))$ be two decorations of $\xi$ and 
\begin{align*}
P_{\mathcal{E},i}&:=\chi_{[0,\infty)}(D^{\partial W}_{\mathcal{E},i}+A_i^\mathcal{E}), \\
P_{\mathcal{E}',i}&:=\chi_{[0,\infty)}(-D^{\partial W}_{\mathcal{E}',i}-A_i^{\mathcal{E}'}), \\
\overline{P}_i&:=\chi_{[0,\infty)}(\overline{D}_{\Xi_i}^{\partial W\times \{0,1\}} + A^{\Xi_i})\quad\mbox{and}\\
P_i&:=\chi_{[0,\infty)}(D_{E\otimes f^*\mathcal{L}_X,i}\dot{\cup}-D_{E'\otimes f^*\mathcal{L}_X,i}+A_i).
\end{align*}
We take a family $(D_{\mathcal{E},u},D_{\mathcal{E}',u},D_{E,u},D_{E',u})_{u\in [0,1]}$ interpolating between the two choices of Dirac operators. Lemma \ref{sfandrho} implies that
\begin{align*}
\rho(W,\Xi_0,f)-\rho(W,\Xi_1,f)=-j_X\mathrm{sf}(D_{E\otimes f^*\mathcal{L}_X,\bullet}\dot{\cup}-D_{E'\otimes f^*\mathcal{L}_X,\bullet}, P_0,P_1),
\end{align*}
while Theorem \ref{datachangingaps} implies that
\begin{align*}
\ind_{APS}(W,\Xi_0,f)-&\ind_{APS}(W,\Xi_1,f)\\
=&\mathrm{sf}(D_{\mathcal{E},\bullet}, P_{\mathcal{E},0},P_{\mathcal{E},1})-\mathrm{sf}(D_{\mathcal{E}',\bullet}, P_{\mathcal{E}',0},P_{\mathcal{E}',1})+\mathrm{sf}(\bar{D}_{\Xi,\bullet}, \bar{P}_{0},\bar{P}_{1}).
\end{align*}
It follows from the construction of $\bar{D}_{\Xi,\bullet}$ that 
\begin{align*}
\mathrm{sf}(\bar{D}_{\Xi,\bullet}, &\bar{P}_{0},\bar{P}_{1})+\mathrm{sf}(D_{E\otimes f^*\mathcal{L}_X,\bullet}\dot{\cup}-D_{E'\otimes f^*\mathcal{L}_X,\bullet}, P_0,P_1)\\
&+\mathrm{sf}(D_{\mathcal{E},\bullet}, P_{\mathcal{E},0},P_{\mathcal{E},1})-\mathrm{sf}(D_{\mathcal{E}',\bullet}, P_{\mathcal{E}',0},P_{\mathcal{E}',1})=0 \quad\mbox{in}\quad K_*(C^*_r(\Gamma)).
\end{align*}
We conclude that $\lambda^\mathcal{S}_{an}(W,\Xi_0,f)-\lambda^\mathcal{S}_{an}(W,\Xi_1,f)=j_X(0)=0$.
\end{proof}

Motivated by Lemma \ref{depeondec}, we sometimes write $\lambda^\mathcal{S}_{an}(W,\xi,f)$ for $\lambda^\mathcal{S}_{an}(W,\Xi,f)$ for some choice of decoration $(W,\Xi,f)$.

\begin{theorem}
\label{themaplambda}
The mapping $\lambda_{an}^\mathcal{S}:\mathcal{S}_*^{geo}(X;\mathcal{L}_X)\to K_*(D^*(\tilde{X})^\Gamma)$
is a well defined isomorphism. Furthermore, it fits into the following commutative diagram:
\begin{equation}
\label{commdiagwithlambda}
\begin{CD}
K_*^{geo}(pt;C^*_r(\Gamma)) @>r >> \mathcal{S}^{geo}_*(X;\mathcal{L}_X)  @> \delta >> K^{geo}_{*+1}(X)  \\
@V\lambda_{an} VV @V\lambda_{an}^\mathcal{S} VV @V\lambda_{an} VV  \\
K_*(C^*_r(\Gamma)) @>j_X >> K_*(D^*(\tilde{X})^\Gamma) @>q_X>> K_*(D^*(\tilde{X})^\Gamma/C^*(\tilde{X})^\Gamma)   \\
\end{CD}
\end{equation}
\end{theorem}

\begin{remark}
The difficult part of the proof is showing that $\lambda^\mathcal{S}_{an}$ is well defined; the proof of that fact will occupy Subsection \ref{bordsubse} and \ref{vecbsubse}. Assuming that $\lambda^\mathcal{S}_{an}$ is well defined, the left part of the diagram \eqref{commdiagwithlambda} commutes since
$$\lambda^\mathcal{S}_{an}(r(M,\mathcal{E}_{C^*_r(\Gamma)}))=j_X\ind_{AS}(M,\mathcal{E}_{C^*_r(\Gamma)})= j_X\lambda_{an}(M,\mathcal{E}_{C^*_r(\Gamma)}).$$
That the right hand side of the diagram commutes follows from Theorem \ref{rhodefandcomq}, which implies that
\begin{align*}
q_X(\lambda^\mathcal{S}_{an}(W,\Xi,f))&=q_X\rho(W,\Xi,f)=f_*\left([M]\cap [E_\C]\right)-f_*\left([M]\cap [E_\C']\right)\\
&=\lambda_{an}(M,E_\C,f)-\lambda_{an}(M,E_\C',f)=\lambda_{an}(\delta(W,\xi,f)).
\end{align*}
The fact that $\lambda^\mathcal{S}_{an}$ is an isomorphism follows from the five lemma once $\lambda^\mathcal{S}_{an}$ is well defined. In \cite{PSrhoInd,PSsignInd}, the delocalized APS-theorem is used to prove commutativity of diagrams of mappings into the analytic surgery exact sequence (see Theorem \ref{anasurgthm}). We only use the delocalized APS-theorem to prove that $\lambda^\mathcal{S}_{an}$ respects bordism.
\end{remark}

\begin{remark}
We will focus on the even-dimensional case, $*=0$. The odd-dimensional case is proved analogously. An alternative route to the isomorphism $\lambda_{an}^\mathcal{S}:\mathcal{S}_1^{geo}(X;\mathcal{L}_X)\to K_1(D^*(\tilde{X})^\Gamma)$ is as follows. View $\tilde{X}\times \field{R}$ as a $\Gamma\times \Z$-presented space with $\tilde{X}\times \field{R}/\Gamma\times \Z=X\times S^1$. One shows that there is a commutative diagram with exact rows 
\small
\[
\begin{CD}
0 @>>>\mathcal{S}_1^{geo}(X,\mathcal{L}_X)@>>>\mathcal{S}_0^{geo}(X\times S^1,\mathcal{L}_{X\times S^1}) @>>>  \mathcal{S}_0^{geo}(X,\mathcal{L}_X) @>>> 0\\
@.@. @V\lambda^\mathcal{S}_{an}VV @ V\lambda^\mathcal{S}_{an}VV@. \\
0 @>>>K_1(D^*(\tilde{X})^\Gamma)@>>>K_0(D^*(\tilde{X}\times \field{R})^{\Gamma\times \field{Z}}) @>>> K_0(D^*(\tilde{X})^\Gamma)@>>> 0,\\
\end{CD}
\]
\normalsize
and then defines $\lambda_{an}^\mathcal{S}:\mathcal{S}_1^{geo}(X;\mathcal{L}_X)\to K_1(D^*(\tilde{X})^\Gamma)$ from this diagram. We note that the map $\mathcal{S}_1^{geo}(X,\mathcal{L}_X)\to \mathcal{S}_0^{geo}(X\times S^1,\mathcal{L}_{X\times S^1})$ is given by product with $[S^1]\in K_1(S^1)$ using $S^1=B\Z$ and \cite[Section 4]{paperI}, while the map $\mathcal{S}_0^{geo}(X\times S^1,\mathcal{L}_{X\times S^1}) \to  \mathcal{S}_0^{geo}(X,\mathcal{L}_X)$ is the forgetful one. The maps in the lower row are more intricate and the reconciliation of this definition of $\lambda_{an}^\mathcal{S}:\mathcal{S}_1^{geo}(X;\mathcal{L}_X)\to K_1(D^*(\tilde{X})^\Gamma)$ with the definition using higher APS-theory would follow from an explicit product formula in analytic surgery theory.
\end{remark}

\subsection{Proof that $\lambda_{an}^\mathcal{S}$ respects bordism}
\label{bordsubse}

We must show that $\lambda_{an}^\mathcal{S}$ is well-defined; in this subsection the bordism relation is considered. As mentioned, the delocalized APS-theorem plays a key role. The precise relationship between the delocalized APS-theorem and bordisms in $\mathcal{S}_*^{geo}(X,\mathcal{L}_X)$ is the content of the next result.

\begin{lemma}
\label{delapsforcycles}
Assume that $(W_0,\xi_0,f_0)$ is a cycle for $\mathcal{S}^{geo}_0(X;\mathcal{L}_X)$, with 
$$\xi_0=(\mathcal{E}_{C^*_r(\Gamma)},\mathcal{E}'_{C^*_r(\Gamma)},E_\C,E_\C',\alpha).$$ 
Assume that 
\begin{itemize}
\item $f_0$ extends to a continuous mapping $g:W_0\to X$; 
\item there exists vector bundles $F_\C,F_\C\to W_0$ extending $E_\C,E'_\C\to \partial W_0$;
\item $\alpha$ extends to an isomorphism over $W_0$:
$$\beta:\mathcal{E}_{C^*_r(\Gamma)}\oplus F_\C'\otimes g^*\mathcal{L}_X\cong \mathcal{E}_{C^*_r(\Gamma)}\oplus F_\C\otimes g^*\mathcal{L}_X.$$ 
\end{itemize}
Then $\lambda^\mathcal{S}_{an}(W_0,\xi_0,f_0)=0$.
\end{lemma}

\begin{proof}
Using Lemma \ref{depeondec}, the general result follows upon proving the vanishing in a specific decoration. The decoration of $\xi_0$ will be chosen in the following manner. We first choose Dirac operators $(D_\mathcal{E}^{W_0},D_{\mathcal{E}'}^{W_0},D_E^{\partial W_0},D_{E'}^{\partial W_0})$. Since the vector bundles $E_\C$ and $E_{\C}'$ extend, we can find Dirac operators $D^{W_0}_F$ and $D_{F'}^{W_0}$ on $S_{W_0}\otimes F_\C$ and $S_{W_0}\otimes F'_\C$ of product type near $\partial W_0$ with boundary operators $D_E^{\partial W_0}$ and respectively $D_{E'}^{\partial W_0}$. As $(\partial W_0, E_\C\otimes f_0^*\mathcal{L}_X)=\partial (W_0, F_\C\otimes g^*\mathcal{L}_X)$ and $(\partial W_0, E_\C'\otimes f_0^*\mathcal{L}_X)=\partial (W_0, F_\C'\otimes g^*\mathcal{L}_X)$ it follows from the bordism invariance of the index that
$$\ind_{C^*_r(\Gamma)}(D_{E\otimes f_0^*\mathcal{L}}^{\partial W_0})=\ind_{C^*_r(\Gamma)}(D_{E'\otimes f_0^*\mathcal{L}}^{\partial W_0})=0.$$
We conclude that there exists smoothing operators $A\in \Psi^{-\infty}_{C^*_r(\Gamma)}(\partial W_0,E_\C\otimes f_0^*\mathcal{L}_X)$ and $A'\in \Psi^{-\infty}_{C^*_r(\Gamma)}(\partial W_0,E_\C'\otimes f_0^*\mathcal{L}_X)$ with
$$D_{E\otimes f_0^*\mathcal{L}}^{\partial W_0}+A\quad\mbox{and}\quad D_{E'\otimes f_0^*\mathcal{L}}^{\partial W_0}+A'\quad\mbox{ invertible}.$$
We choose a decoration of the form 
$$\Xi=(\xi_0,(D_\mathcal{E}^{W_0},D_{\mathcal{E}'}^{W_0},D_E^{\partial W_0},D_{E'}^{\partial W_0}), (A^\mathcal{E},A^{\mathcal{E}'},A\dot{\cup}-A')).$$

Consider the even-dimensional spin$^c$-manifold 
$$Z=W_0\cup_{\partial W_0}\partial W_0\times [0,1]\cup_{\partial W_0}-W_0.$$ 
We define the $C^*_r(\Gamma)$-bundle $\mathcal{F}_{C^*_r(\Gamma)}\to Z$ by 
$$\mathcal{F}_{C^*_r(\Gamma)}=(\mathcal{E}_{C^*_r(\Gamma)}\oplus F_\C'\otimes g^*\mathcal{L}_X)\cup_{\alpha}(\mathcal{E}'_{C^*_r(\Gamma)}\oplus F_\C\otimes g^*\mathcal{L}_X),$$
where we glue over the cylinder using the isomorphism $\alpha$. Since $\alpha$ extends to the isomorphism $\beta$ on $W_0$,  
$$(Z,\mathcal{F}_{C^*_r(\Gamma)})\cong \partial (W_0\times [0,1],(\mathcal{E}_{C^*_r(\Gamma)}\oplus F_\C'\otimes g^*\mathcal{L}_X)\times [0,1]).$$
Here $W_0\times [0,1]$ forms a smooth manifold after ``straightening the angle", see for example \cite[Appendix A]{Rav}. We consider the Dirac operator $D^Z_\mathcal{F}$ given by gluing together $D^{W_0}_\mathcal{E}\oplus D^{W_0}_{F'\otimes g^*\mathcal{L}}$ on $W_0$ with $-D^{W_0}_{\mathcal{E}'}\oplus -D^{W_0}_{F\otimes g^*\mathcal{L}}$ on $-W_0$ via linearly interpolating the boundary operators over the cylinder using the isomorphism $\alpha$. 

We can apply Theorem \ref{gluingthm} to the closed manifold $Z$ cut into several parts by the hyper surface $\partial W_0\times \{0,1\}$; we obtain the identity
\begin{align*}
\ind_{AS}(Z,\mathcal{F}_{C^*_r(\Gamma)})=\ind_{APS}&(W_0,\Xi,f_0)\\
&+\ind_{APS}(D^{W_0}_{F'\otimes g^*\mathcal{L}},A')-\ind_{APS}(D^{W_0}_{F\otimes g^*\mathcal{L}},A).
\end{align*}
On the other hand, bordism invariance of the index implies that the left hand side vanishes. Since 
$$\rho(D^{\partial W_0}_{E}\dot{\cup}-D^{\partial W_0}_{E'},A\dot{\cup}-A')=\rho(D^{\partial W_0}_{E},A)-\rho(D^{\partial W_0}_{E'},A'),$$
it follows that 
\begin{align}
\nonumber
\lambda^\mathcal{S}_{an}(W_0,\xi_0,f_0)=&g_*j_{W_0}\ind_{APS}(D^{W_0}_{F\otimes g^*\mathcal{L}},A)-f_*\rho(D^{\partial W_0}_{E},A)\\
&-g_*j_{W_0}\ind_{APS}(D^{W_0}_{F'\otimes g^*\mathcal{L}},A)+f_*\rho(D^{\partial W_0}_{E'},A).
\label{lambdaborsum}
\end{align}
The identity $\lambda^\mathcal{S}_{an}(W_0,\xi_0,f_0)=0$ follows from \eqref{lambdaborsum} and the delocalized APS-theorem (see Theorem \ref{delapsthemps}, cf. Remark \ref{factoringremark}) which implies that 
\begin{align*}
g_*j_{W_0}\ind_{APS}&(D^{W_0}_{F\otimes g^*\mathcal{L}},A)-f_{0*}\rho(D^{\partial W_0}_{E},A)\\
&=g_*j_{W_0}\ind_{APS}(D^{W_0}_{F'\otimes g^*\mathcal{L}},A)-f_{0*}\rho(D^{\partial W_0}_{E'},A)=0.
\end{align*}
\end{proof}

\begin{lemma}
\label{gluingandlambda}
Let $x_i=(W_i,\xi_i,f_i)$, with $\xi_i= (\mathcal{E}_{C^*_r(\Gamma),i},\mathcal{E}^{\prime}_{C^*_r(\Gamma),i} , E_{\field{C},i}, E_{\field{C},i}^{\prime} ,\alpha_i )$, be cycles for $i=1,2$. Assume that $\partial W_i=(-1)^i Y$, for some spin$^c$-manifold $Y$ on which $ f_1=f_2$, and 
$$(\mathcal{E}_{C^*_r(\Gamma),1},\mathcal{E}^{\prime}_{C^*_r(\Gamma),1} , E_{\field{C},1}, E_{\field{C},1}^{\prime} ,\alpha_1 )|_Y=(\mathcal{E}_{C^*_r(\Gamma),2},\mathcal{E}^{\prime}_{C^*_r(\Gamma),2} , E_{\field{C},2}, E_{\field{C},2}^{\prime} ,\alpha_2)|_Y,$$
as relative cycles for assembly on the pair of spaces $(Y,Y)$. Then
$$\lambda^\mathcal{S}_{an}(x_1)+\lambda^\mathcal{S}_{an}(x_2)=j_X\left(\ind_{AS}(Z,[\mathcal{F}_{C^*_r(\Gamma)}]-[\mathcal{F}'_{C^*_r(\Gamma)}])\right),$$
where $Z:=W_1\cup_YW_2$ and $\mathcal{F}_{C^*_r(\Gamma)},\mathcal{F}_{C^*_r(\Gamma)}'\to Z$ are the $C^*_r(\Gamma)$-bundles obtained from gluing together $\mathcal{E}_{1,C^*_r(\Gamma)}$ respectively $\mathcal{E}'_{1,C^*_r(\Gamma)}$ with $\mathcal{E}_{2,C^*_r(\Gamma)}$ respectively $\mathcal{E}'_{2,C^*_r(\Gamma)}$. 
\end{lemma}

\begin{proof}
To simplify notation, set $E:=E_{\field{C}, 1}=E_{\field{C}, 2}$, $E':=E'_{\field{C}, 1}=E'_{\field{C}, 2}$ and $f:=f_1=f_2$. We choose decorations of the following form
\begin{align*}
\Xi_i=(\xi_i, (D^{W_i}_{\mathcal{E}_i},\,D^{W_i}_{\mathcal{E}'_i},\,&(-1)^i D^{Y}_E, \,(-1)^i D^{Y}_{E'}),(A^{\mathcal{E}_i},A^{\mathcal{E}_i'}, (-1)^iA_i)),
\end{align*}
where $D^{W_1}_{\mathcal{E}_1}=D^{W_2}_{\mathcal{E}_2}$ near $Y$ and $D^{W_1}_{\mathcal{E}_1'}=D^{W_2}_{\mathcal{E}_2'}$ near $Y$. We assume that the decorations are chosen so that $A^{\mathcal{E}_1}=A^{\mathcal{E}_2}$, $A^{\mathcal{E}_1'}=A^{\mathcal{E}_2'}$ and $A_1=A_2$.

The bundles $\mathcal{F}_{C^*_r(\Gamma)}$ and $\mathcal{F}'_{C^*_r(\Gamma)}$ can be equipped with Dirac operators $D^Z_\mathcal{F}$ and $D^Z_{\mathcal{F}'}$ obtained by gluing $D^{W_1}_{\mathcal{E}_1}$ respectively $D^{W_1}_{\mathcal{E}_1'}$ with $D^{W_2}_{\mathcal{E}_2}$ respectively $D^{W_2}_{\mathcal{E}_2'}$. It follows from Theorem \ref{gluingthm} that 
$$\ind_{APS}(D^{W_1}_{\mathcal{E}_1},A^{\mathcal{E}_1})+\ind_{APS}(D^{W_1}_{\mathcal{E}_2},A^{\mathcal{E}_2})=\ind(D^Z_\mathcal{F})=\ind_{AS}(Z,\mathcal{F}),$$
$$\ind_{APS}(D^{W_1}_{\mathcal{E}_1'},A^{\mathcal{E}_1'})+\ind_{APS}(D^{W_1}_{\mathcal{E}_2'},A^{\mathcal{E}_2'})=\ind(D^Z_{\mathcal{F}'})=\ind_{AS}(Z,\mathcal{F}').$$

In a similar fashion,
\begin{align*}
\ind_{APS}(D^{-Y\times [0,1]}_{\Xi_1},A^{\Xi_1})&+\ind_{APS}(D^{Y\times [0,1]}_{\Xi_2},A^{\Xi_2})\\
&=\ind_{AS}(Y\times S^1,(\mathcal{E}_1\oplus E'\otimes f^*\mathcal{L}_X)\times S^1)=0,
\end{align*}
where the last identity follows from bordism invariance of the index. From all these computations, the result follows upon noting that
\begin{align*}
\rho(W_1,\Xi_1,f)+&\rho(W_2,\Xi_2,f)\\
&=f_{*}\left[\rho(D_{E}\dot{\cup}-D_{E'},A)\right] +f_{*}\left[\rho(-D_{E}\dot{\cup}D_{E'},-A)\right]=0.
\end{align*}
\end{proof}

\begin{prop}
\label{bordisnovlam}
If $(W,\xi,f)\sim_{bor} 0$ then $\lambda^\mathcal{S}_{an}(W,\xi,f)=0$ in $K_0(D^*(\tilde{X})^\Gamma)$.
\end{prop}

\begin{proof}
By definition, $x=(W,\xi,f)\sim_{bor} 0$ if and only if there is a cycle $x_0=(W_0,\xi_0,f)$ as in Lemma \ref{delapsforcycles} such that $\partial W=\partial W_0$ and $W\cup_{\partial W} W_0$ being the boundary of a spin$^c$-manifold $Z$, with the bundle data satisfying $\xi|_{\partial W}=\xi_0|_{\partial W_0}$ and that $\xi_{C^*_r(\Gamma)}\cup_{\partial W}\xi_{0,C^*_r(\Gamma)}$ belongs to the image of the restriction mapping $K^0(Z;C^*_r(\Gamma))\to K^0(\partial Z;C^*_r(\Gamma))$. Hence, the bordism invariance of the index, Lemmas \ref{delapsforcycles} and \ref{gluingandlambda} imply that $0=\lambda^\mathcal{S}_{an}(x)+\lambda^\mathcal{S}_{an}(x_0)=\lambda^\mathcal{S}_{an}(x)$.
\end{proof}

\subsection{An intermezzo on vector bundle modification}
\label{vbintermezz}

Recall that if $W$ is a manifold, possibly with boundary, and $V\to W$ is a spin$^c$-bundle of even rank, the vector bundle modification of $W$ is the total space of the sphere bundle $W^V:=S(V\oplus 1_\field{R})$, where $1_\field{R}\to W$ denotes the trivial real line bundle. In regard to this process in general, we refer the reader to \cite{BD,BD2, BHS, HReta}; the references \cite{BHS, HReta} are the most relevant in our context.

To simplify the discussion, we assume that the rank of $V$ is constant and let $2k=\mathrm{rk}(V)$. If $P_V\to W$ denotes the principal $Spin^c(2k)$-bundle of spin$^c$-frames on $V$, one can identify $W^V=P_V\times_{Spin^c(2k)} S^{2k}$. On $S^{2k}$ there is a distinguished $Spin^c(2k)$-equivariant vector bundle $Q_k\to S^{2k}$ called the Bott bundle. The associated vector bundle $Q_V:=P_V\times_{Spin^c(2k)}Q_k\to W^V$ is called the Bott bundle of $V$. For a $C^*$-algebra $B$ and a $B$-bundle $\mathcal{E}_B\to W$ there is an associated vector bundle modified $B$-bundle 
$$\mathcal{E}_B^V:=\pi_V^*(\mathcal{E}_B)\otimes Q_V\to W^V,$$
where $\pi_V:W^V\to W$ denotes the projection. 

Furthermore, there is a $Spin^c(2k)$-equivariant spin$^c$-structure given by a Clifford bundle $S_k\to S^{2k}$ such that $S_{W^V}=\pi_V^*S_W\hat{\otimes} S_{W^V/W}$ as Clifford bundles, where $S_{W^V/W}:=P_V\times_{Spin^c(2k)}S_k$. By construction, we can identify 
$$C^\infty(W^V,S_{W^V}\otimes \mathcal{E}^V_B)= [C^\infty(P_V\times S^{2k}, (\pi_V^*S_W\hat{\boxtimes} S_k)\otimes (\pi_V^*\mathcal{E}_B\boxtimes Q_k))]^{Spin^c(2k)},$$ 
where $[\mathcal{V}]^{Spin^c(2k)}$ denotes the $Spin^c(2k)$-invariant part of a $Spin^c(2k)$-representation $\mathcal{V}$. In fact, after choosing a suitable $Spin^c(2k)$-invariant Laplacian, we can carry out the same identification for the Sobolev $C^*_r(\Gamma)$-modules for any $s\in \field{N}$:
\begin{equation}
\label{hsident}
H^s(W^V,S_{W^V}\otimes \mathcal{E}^V_B)=\big[H^s(P_V\times S^{2k}, (\pi_V^*S_W\hat{\boxtimes} S_k)\otimes (\pi_V^*\mathcal{E}_B\boxtimes Q_k))\big]^{Spin^c(2k)}.
\end{equation}

On the Clifford bundle $S_k\otimes Q_k$ there is a $Spin^c(2k)$-equivariant spin$^c$-Dirac operator $D_Q$ such that the even part $D_Q^+$ has a one-dimensional kernel, giving the trivial $Spin^c(2k)$-representation, and $\ker D^-_Q=0$. We let $e_Q\in \Psi^{-\infty}(S^{2k},S_k\otimes Q_k)$ denote the projection onto this kernel, it is smoothing because of elliptic regularity. Since $\ker D_Q^+$ is $Spin^c(2k)$-invariant, so is $e_Q$. For details, see \cite[Proposition 3.11]{BHS}.

The next definition and Propositions \ref{eqprop}, \ref{decomposingmoduledir}, \ref{someinnerprod}, and \ref{decomposingmoddir} should be compared with \cite[Propositions 3.6 and 3.11]{BHS}.

\begin{define}[Vector bundle modification of Dirac operators]
Let $B$ be a $C^*$-algebra and $\mathcal{E}_B\to W$ a $B$-bundle over a spin$^c$-manifold $W$ equipped with a twisted spin$^c$-Dirac operator $D_\mathcal{E}$. There is an associated vector bundle modified twisted spin$^c$-Dirac operator on $\mathcal{E}_B^V\to W^V$ given by
$$D^V_\mathcal{E}:=\left(D_\mathcal{E}\hat{\otimes} 1+1\hat{\otimes} D_Q\right)|_{C^\infty(W^V, S_{W^V}\otimes \mathcal{E}_B^V)}.$$
\end{define}

One verifies in local coordinates that the vector bundle modification of Dirac operators produces a well defined Dirac operator. Using the projection $e_Q$, we will decompose the Sobolev space $H^s(W^V,S_{W^V}\otimes \mathcal{E}^V_B)$. The proof of the next proposition follows from the construction of $e_Q$.

\begin{prop}
\label{eqprop}
The operator $1\boxtimes e_Q$ acting on $C^\infty(P_V\times S^{2k}, (\pi_V^*S_W\hat{\boxtimes} S_k)^+\otimes (\pi_V^*\mathcal{E}_B\boxtimes Q_k))$ restricted to the $Spin^c(2k)$-invariant part: 
$$e_Q^V:=(1\boxtimes e_Q)|_{C^\infty(W^V,S_{W^V}\otimes \mathcal{E}^V_B)},$$
defines a projection $e^V_Q$ in any of the Sobolev $C^*_r(\Gamma)$-modules $H^s(W^V,S_{W^V}\otimes \mathcal{E}^V_B)$ satisfying 
$$[D^V_\mathcal{E},e_Q^V]=0 \quad\mbox{on}\quad H^s(W^V,S_{W^V}\otimes \mathcal{E}^V_B)\;\forall s\in \field{N}_{>0}.$$
\end{prop}

\begin{define}
We define the complemented $B$-Hilbert module:
$$\mathpzc{E}_B^s:=(1-e_Q^V)H^s(W^V,S_{W^V}\otimes \mathcal{E}^V_B)\subseteq H^s(W^V,S_{W^V}\otimes \mathcal{E}^V_B).$$
For simplicity, we set $\mathpzc{E}_B:=\mathpzc{E}_B^0$. 
\end{define}

\begin{prop}
\label{decomposingmoduledir}
Let $W$ be a spin$^c$-manifold, $V\to W$ a spin$^c$-vector bundle of even rank and $\mathcal{E}_B\to W$ a $B$-bundle. There is a $C^\infty(W)$-linear isomorphism of $B$-Hilbert modules
$$e^V_QH^s(W^V,S_{W^V}\otimes \mathcal{E}^V_B)\cong H^s(W,S_W\otimes \mathcal{E}_B),$$
which is graded if $W$ is even-dimensional. In particular, for $s\in \field{N}$, there is an orthogonal direct sum decomposition of $B$-Hilbert modules
\begin{equation}
\label{decompositingl2forbundle}
H^s(W^V,S_{W^V}\otimes \mathcal{E}^V_B)\cong H^s(W,S_W\otimes \mathcal{E}_B)\oplus \mathpzc{E}_B^s.
\end{equation}
The decomposition \eqref{decompositingl2forbundle} respects the left action of $C^\infty(W)$ and is graded if $W$ is even-dimensional.
\end{prop}

\begin{proof}
The isomorphism $e^V_QH^s(W^V,S_{W^V}\otimes \mathcal{E}^V_B)\cong H^s(W,S_W\otimes \mathcal{E}_B)$ is constructed as the composition of the following chain of (graded) isomorphisms:
\begin{align*}
e^V_QH^s(W^V,S_{W^V}\otimes \mathcal{E}^V_B)&\cong \ker \left(\left(1\otimes D_Q\right)|_{H^s(W^V,S_{W^V}\otimes \mathcal{E}^V_B)}\right)\\
&\cong H^s(W,S_W\otimes \mathcal{E}_B)\otimes \ker D_Q\cong H^s(W,S_W\otimes \mathcal{E}_B).
\end{align*}
Equation \eqref{decompositingl2forbundle} follows from the existence of this isomorphism since $e_Q^V$ preserves the Sobolev scale by Proposition \ref{eqprop}. The decomposition \eqref{decompositingl2forbundle} respects the left $C^\infty(W)$-action since $e_Q^V$ commutes with the left $C^\infty(W)$-action.
\end{proof}

\begin{remark}
It is clear that \eqref{decompositingl2forbundle} respects the $C(W)$-action for $s=0$ and the $Lip(W)$-action for $s\leq 1$.
\end{remark}

\begin{prop}
\label{someinnerprod}
For any $f\in H^1(W^V,S_{W^V}\otimes \mathcal{E}^V_B)$,
\begin{enumerate}
\item 
\begin{align*}
\langle (D_\mathcal{E}&\hat{\otimes} 1)f,(1\hat{\otimes}D_Q)f\rangle_{ L^2(P_V\times S^{2k}, (\pi_V^*S_W\hat{\boxtimes} S_k)\otimes (\pi_V^*\mathcal{E}_B\boxtimes Q_k))}\\
&+\langle (1\hat{\otimes}D_Q)f,(D_\mathcal{E}\hat{\otimes} 1)f\rangle_{ L^2(P_V\times S^{2k}, (\pi_V^*S_W\hat{\boxtimes} S_k)\otimes (\pi_V^*\mathcal{E}_B\boxtimes Q_k))}=0
\end{align*}
\item  The inequality 
$$\langle (1\hat{\otimes} D_Q)f,(1\hat{\otimes} D_Q)f\rangle_{ L^2(W^V,S_{W^V}\otimes \mathcal{E}^V_B)}\geq c\langle (1-e_Q^V)f, (1-e_Q^V)f\rangle_{\mathpzc{E}_B}$$
holds in $B$ for a $c>0$ only depending on $k$.
\end{enumerate}
\end{prop}

We omit the proof of Proposition \ref{someinnerprod} as it reduces to well known computations by localizing to a coordinate chart. The number $c>0$ is the spectral gap of the Laplacian type operator $D_Q^2$ on $S^{2k}$ above zero. The main tool when dealing with vector bundle modifications will be the following Lemma which decomposes operators into a ``nice" part and a ``bad" part.

\begin{lemma}
\label{decomposingmoddir}
Following the notation of Proposition \ref{decomposingmoduledir}, we let $D_\mathcal{E}$ denote a twisted spin$^c$-Dirac operator on $\mathcal{E}_B$. The direct sum decomposition \eqref{decompositingl2forbundle} is compatible with $D^V_\mathcal{E}$ in the sense that there is a densely defined regular (odd) operator $D_\mathpzc{E}:\mathpzc{E}_B\to \mathpzc{E}_B$ with domain $\mathpzc{E}_B^1$ such that 
\begin{equation}
\label{decomposiingthatd}
D^{V}_\mathcal{E}=D_\mathcal{E}\oplus D_\mathpzc{E}.
\end{equation}
on the dense subspace 
$$H^1(W,S_W\otimes \mathcal{E}_B)\oplus \mathpzc{E}_B^1=H^1(W^V,S_{W^V}\otimes \mathcal{E}^V_B)\subseteq L^2(W^V,S_{W^V}\otimes \mathcal{E}^V_B).$$ 
Moreover, the decomposition \eqref{decomposiingthatd} also has the following property: for any $f\in \mathpzc{E}_B^1\subseteq H^1(W^V,S_{W^V}\otimes \mathcal{E}^V_B)$ there exists a $c$ (depending only on $k$) such that 
\begin{equation}
\label{postivitiy}
\langle D^V_\mathcal{E}f,D^V_\mathcal{E}f\rangle_{ L^2(W^V,S_{W^V}\otimes \mathcal{E}^V_B)}\geq c\langle f,f\rangle_{\mathpzc{E}_B}\quad\mbox{in $B$}.
\end{equation}
\end{lemma}

\begin{proof}
The decomposition \eqref{decomposiingthatd} follows since $e_Q^V$ respects the Sobolev scale and commutes with $D^{V}_\mathcal{E}$; a fact proved in Proposition \ref{eqprop}. To prove \eqref{postivitiy}, we note that we can identify $f\in \mathpzc{E}_B^1\subseteq H^1(W^V,S_{W^V}^+\otimes \mathcal{E}^V_B)$ with a $Spin^c(2k)$-invariant section of 
\begin{equation}
\label{bighone}
H^1(P_V\times S^{2k}, (\pi_V^*S_W\hat{\boxtimes} S_k)\otimes (\pi_V^*\mathcal{E}_B\boxtimes Q_k))
\end{equation}
on which $D_\mathcal{E}\hat{\otimes} 1$ and $1\hat{\otimes}D_Q$, by Proposition \ref{someinnerprod}, anti commute in form sense. Again, by Proposition \ref{someinnerprod}, we have that $1\otimes D_Q^2\geq c$ in the $B$-valued form sense on the space 
$$(1-1\otimes e_Q)H^1(P_V\times S^{2k}, (\pi_V^*S_W\hat{\boxtimes} S_k)\otimes (\pi_V^*\mathcal{E}_B\boxtimes Q_k)).$$
It follows from these observations that  
\begin{align*}
\langle D^V_\mathcal{E}f,D^V_\mathcal{E}f&\rangle_{ L^2(W^V,S_{W^V}^-\otimes \mathcal{E}^V_B)}\\
&=\langle (D_\mathcal{E}\hat{\otimes} 1)f,(D_\mathcal{E}\hat{\otimes} 1)f\rangle_{L^2(P_V\times S^{2k}, (\pi_V^*S_W\hat{\boxtimes} S_k)\otimes (\pi_V^*\mathcal{E}_B\boxtimes Q_k))}\\
&\qquad +\langle (1\hat{\otimes}D_Q)f,(1\hat{\otimes}D_Q)f\rangle_{ L^2(P_V\times S^{2k}, (\pi_V^*S_W\hat{\boxtimes} S_k)\otimes (\pi_V^*\mathcal{E}_B\boxtimes Q_k))}\\
&\geq \langle (1\hat{\otimes}D_Q)f,(1\hat{\otimes}D_Q)f\rangle_{ L^2(P_V\times S^{2k}, (\pi_V^*S_W\hat{\boxtimes} S_k)\otimes (\pi_V^*\mathcal{E}_B\boxtimes Q_k))}\\
&\geq c\langle f,f\rangle_{ L^2(P_V\times S^{2k}, (\pi_V^*S_W\hat{\boxtimes} S_k)\otimes (\pi_V^*\mathcal{E}_B\boxtimes Q_k))}= c\langle f,f\rangle_{\mathpzc{E}_B}
\end{align*}
\end{proof}

\begin{define}[Vector bundle modification of smoothing operators]
Let $M$ be a closed manifold, $\mathcal{E}_B\to M$ a $B$-bundle and  $A\in \Psi^{-\infty}_B(M,\mathcal{E}_B)$ a smoothing operator. The vector bundle modified smoothing operator is defined to be 
$$A^V=(A\hat{\otimes}e_Q)|_{L^2(W^V, S_{W^V/W}\otimes \mathcal{E}_B^V)}.$$
\end{define}

It was shown in \cite[Lemma 2.9]{paperII} that the vector bundle modification of a trivializing operator is again a trivializing operator. The proof of the next proposition follows directly from the same method of proof as in Lemma \ref{decomposingmoddir}.

\begin{prop}
Let $W$, $\mathcal{E}_B\to W$ and $V\to W$ be as in Lemma \ref{decomposingmoddir} with $\partial W=\emptyset$. If $A\in \Psi^{-\infty}_B(W,\mathcal{E}_B\otimes S_W)$, then under the decomposition \eqref{decompositingl2forbundle}, the smoothing operator $A^V$ decomposes as  $A^V=A\oplus 0_{\mathpzc{E}_B}$.
\end{prop}

\begin{prop}
\label{vbundmodofaps}
Assume that $W$ is an even-dimensional spin$^c$-manifold and that $\mathcal{E}_B\to W$ is a $B$-bundle. Let $V\to W$ be a spin$^c$-vector bundle of even rank. For any a twisted spin$^c$-Dirac operator $D_\mathcal{E}$ on $\mathcal{E}_B$ of product type near $\partial W$ and any trivializing operator $A\in \Psi^{-\infty}_B(\partial W, S_{\partial W}\otimes  \mathcal{E}_B|_{\partial W})$, 
$$\ind_{APS}(D_\mathcal{E},A)=\ind_{APS}(D_\mathcal{E}^V,A^V)\quad\mbox{in}\quad K_0(C^*_r(\Gamma)).$$
\end{prop}

\begin{proof}
Lemma \ref{decomposingmoddir} implies that $D_\mathcal{E}^{V,\infty}+A^V_\infty$ can be realized as a direct sum of the operators
$$D_\mathcal{E}^{V,\infty}+A^V_\infty=(D_\mathcal{E}^\infty+A_\infty)\oplus D_\mathpzc{E}^\infty,$$
where $D_\mathpzc{E}^\infty:=(1-e^{V_\infty}_Q)D^{W_\infty^V}_{\mathcal{E}^V}$. The operators $D_\mathcal{E}^{V,\infty}$ and $D_\mathcal{E}^{\infty}$ are regular and self-adjoint, therefore $D_\mathpzc{E}^\infty$ must also be regular and self-adjoint. It follows from Lemma \ref{decomposingmoddir} (more precisely the inequality \eqref{postivitiy}) that $D_\mathpzc{E}^\infty$ is invertible; hence
$$\ind_{APS}(D_\mathcal{E}^V,A^V)=\ind_{C^*_r(\Gamma)}(D_\mathcal{E}^\infty+A_\infty)+ \ind_{C^*_r(\Gamma)}D_\mathpzc{E}^\infty=\ind_{APS}(D_\mathcal{E},A).$$
\end{proof}

We now turn to closed manifolds and higher $\rho$-invariants. Our construction is based on the \cite[Section 8]{HReta} (in particular, see the proof of \cite[Theorem 8.18]{HReta}). We note that in this case, we need an analog of Lemma \ref{decomposingmoddir} for the equivariant Dirac operators on $L^2(\tilde{M})$ for a covering $\Gamma\to \tilde{M}\to M$. On manifolds like $\tilde{M}$, we use the Sobolev spaces coming from a suitable $\Gamma$-equivariant Laplacian type operator.

\begin{lemma}
Assume that $M$ is a closed spin$^c$-manifold, $f:M\to B\Gamma$ a continuous mapping, $E\to M$ a vector bundle, and $D^M_E$ a Dirac operator on $S_M\otimes E$. Let $V\to M$ be a spin$^c$-vector bundle of even rank. The projection $e^V_Q\otimes_\pi \mathrm{id}_{\ell^2(\Gamma)}$ respects the Sobolev scale $H^s(\tilde{M}^V,S_{\tilde{M}^V}\otimes \tilde{E}^V)$ and decomposes
\begin{equation}
\label{hsdecomp}
H^s(\tilde{M}^V,S_{\tilde{M}^V}\otimes \tilde{E}^V)=H^s(\tilde{M},S_{\tilde{M}}\otimes \tilde{E})\oplus \mathpzc{E}^s_{L^2},
\end{equation}
in a (graded) $C^\infty_c(\tilde{M})$-linear way. Under the decomposition \eqref{hsdecomp}, 
$$\tilde{D}^{M^V}_{E^V}=\tilde{D}^M_E\oplus D_{\mathpzc{E}_{L^2}},$$
where $D_{\mathpzc{E}_{L^2}}$ is a self-adjoint invertible operator with $\mathrm{dom}(D_{\mathpzc{E}_{L^2}})= \mathpzc{E}^1_{L^2}$.
\end{lemma} 

The proof follows the same lines as the one of Lemma \ref{decomposingmoddir}. The fact that the projection $e^V_Q\otimes_\pi \mathrm{id}_{\ell^2(\Gamma)}$ respects the Sobolev scale follows since it coincides with the orthogonal projection onto the closed subspace
$$H^s(\tilde{M},S_{\tilde{M}}\otimes \tilde{E})\otimes \ker D_Q\subseteq H^s(\tilde{M}^V,S_{\tilde{M}^V}\otimes \tilde{E}^V).$$

Let $\gamma_k$ denote the grading operator on $S_k\otimes Q_k$ and define an operator $J_Q$ by taking $J_Q:=i\gamma_kD_Q|D_Q|^{-1}\in \mathcal{B}(L^2(S^{2k}, S_k\otimes Q_k))$ on $\ker(D_Q)^{\bot}$ and $J_Q|_{\ker D_Q}=0$. Define 
$$J:=(1\hat{\otimes} J_Q)|_{\mathpzc{E}_{L^2}}\in \mathcal{B}(\mathpzc{E}_{L^2}).$$

\begin{lemma}
\label{thejlem}
The self-adjoint operator $J\in \mathcal{B}(\mathpzc{E}_{L^2})$ has the following properties
\begin{enumerate}
\item $J^2=1$ (i.e., $J$ is an involution);
\item $J$ preserves the scale $\mathpzc{E}^s_{L^2}$;
\item $JD_{\mathpzc{E}_{L^2}}+D_{\mathpzc{E}_{L^2}}J=0$ on $\mathpzc{E}^1_{L^2}$;
\item $J\in D^*(\tilde{M},\mathpzc{E}_{L^2})^\Gamma$ and commutes with the $C_0(\tilde{M})$-action.
\end{enumerate}
\end{lemma}

\begin{proof}
It is clear that $J^2=1$ because $J_Q^2= 1-e_Q$. Moreover, $J_Q$ formally anti commutes with $D_Q$, and $1\hat{\otimes} J_Q$ formally anti commutes with $\tilde{D}_M^E\hat{\otimes} 1$; hence, $J$ formally anitcommutes with $D_{\mathpzc{E}_{L^2}}$. It also follows that $J$ preserves the scale $\mathpzc{E}^s_{L^2}$ as it commutes with a suitable choice of Laplacian operator on $\tilde{M}^V$ coming from $(\tilde{D}_M^E)^2\otimes 1+1\otimes D_Q^2$. We can (by a density argument) conclude that $JD_{\mathpzc{E}_{L^2}}+D_{\mathpzc{E}_{L^2}}J=0$ on $\mathpzc{E}^1_{L^2}$.

It remains to prove that $J\in D^*(\tilde{M},\mathpzc{E}_{L^2})^\Gamma$. Since $e^V_Q\otimes_\pi \mathrm{id}_{\ell^2(\Gamma)}$ commutes with the $C^\infty_c(\tilde{M})$-action on $\mathpzc{E}_{L^2}$, it is clear that $J$ commutes with the $C_0(\tilde{M})$-action and is pseudo-local with finite propagation; in fact, it has zero propagation.
\end{proof}

\begin{lemma}
\label{jandplemma}
Suppose $p\in D^*(\tilde{X})^\Gamma$ be a projection. Then, if there exists a self-adjoint involution $J\in D^*(\tilde{X})^\Gamma$ such that
\begin{itemize}
\item $J$ commutes with the $C_0(\tilde{X})$-action;
\item $Jp+pJ=J$, 
\end{itemize} 
then $[p]=0\quad\mbox{in}\quad K_0(D^*(\tilde{X})^\Gamma).$
\end{lemma}

\begin{proof}
Recall that $\mathcal{H}_X$ denotes the $\Gamma$-equivariant $X$-module used to define $D^*(\tilde{X})^\Gamma$. We set $q=(1+J)/2$; it is a projection in $D^*(\tilde{X},\mathcal{H}_X)^\Gamma$. Consider the function $Q\in C([0,1],D^*(\tilde{X},\mathcal{H}_X)^\Gamma)$ given by 
$$Q(t)=\frac{tJ+\sqrt{1-t^2}(2p-1)+1}{2}.$$
Since $J(2p-1)=-(2p-1)J$, the operator $tJ+\sqrt{1-t^2}(2p-1)$ is a symmetry for any $t$; hence, $Q(t)$ is projection valued. Furthermore, $Q(0)=p$ and $Q(1)=q$. We conclude that $[p]=[q]$ in $K_0(D^*(\tilde{X},\mathpzc{H}_X)^\Gamma)$. Since the conditions on $J$ are symmetric under $J\mapsto -J$, the same argument applies to $1-q=(1-J)/2$ and so $[p]=[1-q]$ in $K_0(D^*(\tilde{X},\mathpzc{H}_X)^\Gamma)$. 

After possibly tensoring with $\ell^2(\field{N})$, we can assume that $q\mathpzc{H}_X$ or $(1-q)\mathpzc{H}_X$ is ample (i.e., that it is a $\Gamma$-equivariant $X$-module). For simplicity, we assume $q\mathpzc{H}_X$ satisfies this assumption. Under the isomorphism $K_0(D^*(\tilde{X},\mathpzc{H}_X)^\Gamma)\cong K_0(D^*(\tilde{X},q\mathpzc{H}_X)^\Gamma)$ we can identify $[q]$ with $[1]$. The result follows since $[1]=0\in K_0(D^*(\tilde{X},\mathpzc{H}_X)^\Gamma)$; a fact that can be proved using a ``swindle" argument.
\end{proof}

\begin{lemma}
\label{rhoandvb}
Let $M$ be a closed spin$^c$-manifold, $f:M\to X$ be a continuous map, $V\to M$ be a spin$^c$-vector bundle of even rank, $E\to M$ be a complex vector bundle, $D_E$ be a twisted spin$^c$-Dirac operator on $E$, and $A\in \Psi^{-\infty}_{C^*_r(\Gamma)}(M,S_M\otimes E\otimes f^*\mathcal{L}_X)$ be a trivializing operator. Then, $\rho(D_{E}^V,A^V)=\rho(D_{E},A)\quad\mbox{in}\quad K_*(D^*(\tilde{M})^\Gamma)$.
\end{lemma}

\begin{proof}
It follows from Lemma \ref{decomposingmoddir} that 
$$\chi_{[0,\infty)}(\tilde{D}_{E}^V+\tilde{A}^V)=\chi_{[0,\infty)}(\tilde{D}_{E}+\tilde{A})\oplus \chi_{[0,\infty)}(D_{\mathpzc{E}_{L^2}}).$$
Since 
$$\chi_{[0,\infty)}(\tilde{D}_{E}+\tilde{A})\in D^*(\tilde{M}, L^2(\tilde{M},S_{\tilde{M}}\otimes \tilde{E}))^\Gamma$$ 
and 
$$\chi_{[0,\infty)}(\tilde{D}_{E}^V+\tilde{A}^V)\in D^*(\tilde{M}, L^2(\tilde{M}^V,S_{\tilde{M}^V}\otimes \tilde{E}^V))^\Gamma,$$
we can consider $\chi_{[0,\infty)}(D_{\mathpzc{E}_{L^2}}) \in D^*(\tilde{M}, \mathpzc{E}_{L^2})^\Gamma$. Furthermore, under suitable identifications,
$$\rho(D_{E}^V,A^V)=\rho(D_{E},A)+[\chi_{[0,\infty)}(D_{\mathpzc{E}_{L^2}})].$$
The result follows by noting that Lemmas \ref{thejlem} and \ref{jandplemma} imply that $[\chi_{[0,\infty)}(D_{\mathpzc{E}_{L^2}})]=0$ in $K_*(D^*(\tilde{M})^\Gamma)$.
\end{proof}

\subsection{Proof that $\lambda_{an}^\mathcal{S}$ respects vector bundle modification}
\label{vecbsubse}

Finally, we turn to proving that $\lambda^\mathcal{S}_{an}$ respects vector bundle modification. In order to do so, we will make use of a special choice of decoration for the vector bundle modified cycle. 

\begin{define}[Vector bundle modification of decorations]
Let $(W,\xi,f)$ denote a cycle and 
$$\Xi=(\xi,(D_{\mathcal{E}},D_{\mathcal{E}'},D_E,D_{E'}),(A^\mathcal{E},A^{\mathcal{E}'},A))$$ 
a decoration of $\xi$. For a spin$^c$-vector bundle $V\to W$ of even rank, we define the following decoration of $\xi^V$:
$$\Xi^V:=(\xi,(D_{\mathcal{E}}^V,D_{\mathcal{E}'}^V,D_E^V,D_{E'}^V),(A^\mathcal{E}_V,A^{\mathcal{E}'}_V,A^V)),$$ 
where (for simplicity) we have omitted the notation indicating restriction of the bundle $V$ to $\partial W$.
\end{define}

It follows from \cite[Lemma 2.9]{paperII} that $\Xi^V$ is a well defined decoration of $\xi^V$.

\begin{prop}
\label{vbmodofaps}
Let $(W,\Xi,f)$ be a decorated cycle and $V\to W$ a spin$^c$-vector bundle of even rank. Then, $\ind_{APS}(W,\Xi,f)=\ind_{APS}(W^V,\Xi^V,f^V)$.
\end{prop}

\begin{proof}
The proof follows by applying Proposition \ref{vbundmodofaps} to the terms involved in the definition of the APS-index of a decorated cycle (see Definition \ref{defininginvaria}).
\end{proof}

\begin{prop}
\label{vbmodofrho}
Let $(W,\Xi,f)$ be a decorated cycle and $V\to W$ a spin$^c$-vector bundle of even rank. Then, $\rho(W,\Xi,f)=\rho(W^V,\Xi^V,f^V)$.
\end{prop}

\begin{proof}
The proof follows from the identity $(\tilde{D}_{E}\dot{\cup}\tilde{D}_{E'})^V=\tilde{D}_{E}^V\dot{\cup}\tilde{D}_{E'}^V$ and a direct application of Lemma \ref{rhoandvb}.
\end{proof}

\begin{cor}
The mapping $\lambda^\mathcal{S}_{an}:\mathcal{S}^{geo}_*(X,\mathcal{L}_X)\to K_*(D^*(\tilde{X})^\Gamma)$ is well defined.
\end{cor}

\begin{proof}
The mapping $\lambda^\mathcal{S}_{an}$ trivially respects disjoint union/direct sum and vanishes on degenerate cycles. By Proposition \ref{bordisnovlam}, $\lambda^\mathcal{S}_{an}$ respects bordism. By combining Proposition \ref{vbmodofaps} with Proposition \ref{vbmodofrho}, it follows that $\lambda^\mathcal{S}_{an}$ respects vector bundle modification.
\end{proof}

\begin{remark}
\label{delapsandszero}
We remark that $\mathcal{S}^{geo}_*(X,\mathcal{L}_X)=K_*(D^*(\tilde{X})^\Gamma)=0$ holds if and only if $\lambda^\mathcal{S}_{an}=0$. Such a statement is equivalent to an extended version of the delocalized APS-index theorem:
$$\lambda^\mathcal{S}_{an}(W,\Xi,f)=0\quad\Leftrightarrow \quad j_X\ind_{APS}(W,\Xi,f)=\rho(W,\Xi,f).$$
If $\Gamma$ has torsion elements, the failure of the free assembly map to be an isomorphism is an obstruction to this extended version of the delocalized APS-index theorem.
\end{remark}

\section{Delocalized Chern characters} \label{sectionDeloc}

In this section, Chern characters on the geometric model for the analytic surgery group are considered. Following the modus operandi in this series of paper, a delocalized Chern character should commute with a specified choice of Chern character on $K_*(B\Gamma)$ and $K_*(C^*_r(\Gamma))$. As such, it will depend on the choice of a dense subalgebra $\mathcal{A}\subseteq C^*_r(\Gamma)$ as well as the choice of suitable homology theory of $\mathcal{A}$. The target group for the delocalized Chern character will depend heavily on these choices. Indeed, the delocalized Chern character (being a relative construction) depends on the precise form of the Chern character on $\mathcal{A}$ at the level of cycles. As such, the choice of homology theory plays an important role in its very definition. The choice we make is a topological version of noncommutative de Rham homology. We follow \cite{Kaast}. This homology theory has previously been used in the context of higher index theory in for instance \cite{lottsuper,lotthighereta,WahlAPSForCstarBun}. Related constructions have been carried out in \cite{bcfete,lemopf,moriyoshi,moriyoshipiazza}.

\subsection{Noncommutative de Rham homology} 
\label{leshom}

We begin with a rather long list of notation. Let $\mathcal{A}$ denote a unital algebra. The universal derivation on $\mathcal{A}$ is given by $\mathrm{d}_\mathcal{A}:\Omega_0(\mathcal{A})=\mathcal{A}:\to \Omega_1(\mathcal{A}):=\mathcal{A}\otimes (\mathcal{A}/\field{C}1_\mathcal{A})$, $a\mapsto \mathrm{d}a=1\otimes a$.  From the universal derivation one constructs the universal $\Z$-graded differential algebra of $\mathcal{A}$ as $\Omega_*(\mathcal{A}):=\bigoplus_{k=0}^\infty \Omega_k(\mathcal{A})$ where $\Omega_k(\mathcal{A}):=\Omega_1(\mathcal{A})^{\otimes_\mathcal{A} k}=\mathcal{A}\otimes (\mathcal{A}/\field{C}1_\mathcal{A})^{\otimes_\field{C} k}$ for $k>0$. The derivation on $\Omega_*(\mathcal{A})$ is defined from $\mathrm{d}_\mathcal{A}$ via the Leibniz rule, by an abuse of notation we denote the derivation on $\Omega_*(\mathcal{A})$ by $\mathrm{d}_\mathcal{A}$. We define $\Omega_*^{ab}(\mathcal{A}):=\Omega_*(\mathcal{A})/[\Omega_*(\mathcal{A}),\Omega_*(\mathcal{A})]$ where all commutators are graded. If $\mathcal{A}$ is a unital Fr\'echet  $*$-algebra, after replacing all algebraic tensor products by projective tensor products we arrive at the universal $\Z$-graded differential Fr\'echet  algebra $\hat{\Omega}_*(\mathcal{A})$ of $\mathcal{A}$. We also define $\hat{\Omega}_*^{ab}(\mathcal{A}):=\hat{\Omega}_*(\mathcal{A})/\overline{[\hat{\Omega}_*(\mathcal{A}),\hat{\Omega}_*(\mathcal{A})]}$. For details on topological homology, see \cite{helembook,jotaylor}.

For an algebra $\mathcal{A}$, we let $\tilde{\mathcal{A}}$ denote its unitalization. If $\mathcal{A}$ is a Fr\'echet  $*$-algebra, then so is $\tilde{\mathcal{A}}$. 

\begin{define}
The de Rham homology $H_*^{dR}(\mathcal{A})$ of an algebra $\mathcal{A}$ is defined as the homology $H_*(\Omega_*^{ab}(\tilde{\mathcal{A}}))$. The topological de Rham homology $\hat{H}_*^{dR}(\mathcal{A})$ of a Fr\'echet  algebra $\mathcal{A}$ is defined as the topological homology of $\hat{\Omega}_*^{ab}(\tilde{\mathcal{A}})$ (i.e., closed cycles modulo the closure of the exact cycles).
\end{define}

We use the notation $C_*(\mathcal{A})$ for the Hochschild complex, $C_*^\lambda(\mathcal{A})$ for the cyclic complex and $b$ for the boundary operator on both of these complexes. We use $HH_*(\mathcal{A})$ and $HC_*(\mathcal{A})$ to denote the Hochschild homology groups respectively the cyclic homology groups of $\mathcal{A}$. For details, see \cite{cuncycenc,Kaast}. The construction of the $S$-operator on cyclic homology is found in for instance \cite[Page 11]{cuncycenc}.

\begin{prop}[Theorem 2.15 of \cite{Kaast}]
\label{hdrandhcyc}
The map
\begin{equation}
\label{omegatocmodb}
\Omega_k^{ab}(\tilde{\mathcal{A}})\ni a_0\mathrm{d} a_1\cdots \mathrm{d} a_k\mapsto a_0\otimes a_1\otimes \cdots a_k\in C_k^\lambda(\mathcal{A})/\mathrm{Im}(b)
\end{equation}
induces an isomorphism 
$$H_*^{dR}(\mathcal{A})\cong \mathrm{Im}(S:HC_{*+2}(\mathcal{A})\to HC_*(\mathcal{A}))=\ker(B:HC_*(\mathcal{A})\to HH_{*+1}(\mathcal{A})).$$
\end{prop}

\begin{proof}
Excision (of split short exact sequences) implies that the reduced cyclic homology $\overline{HC}(\tilde{\mathcal{A}})$ of $\tilde{\mathcal{A}}$ coincides with the cyclic homology of $\mathcal{A}$. It follows from \cite[Theorem 2.15]{Kaast} that the mapping \eqref{omegatocmodb} induces an isomorphism 
$$H_*^{dR}(\mathcal{A})=H_*(\Omega^{ab}_*(\tilde{\mathcal{A}}))\cong \mathrm{Im}(\bar{S}:\overline{HC}_{*+2}(\tilde{\mathcal{A}})\to \overline{HC}_*(\tilde{\mathcal{A}})).$$
By naturality of the $S$-operator, the quotient mapping induces an isomorphism $\mathrm{Im}(\bar{S}:\overline{HC}_{*+2}(\tilde{\mathcal{A}})\to \overline{HC}_*(\tilde{\mathcal{A}}))\cong \mathrm{Im}(S:HC_{*+2}(\mathcal{A})\to HC_*(\mathcal{A}))$ proving the Lemma. 
\end{proof}

We will restrict our attention to Fr\'echet  $*$-algebra completions $\mathcal{A}(\Gamma)$ of $\C[\Gamma]$. To simplify notations, we write $\tilde{\mathcal{A}}(\Gamma)$ and $\tilde{\C}[\Gamma]$ for their respective unitalizations. By functoriality, there is a differential graded map 
$$j_{\Omega,\mathcal{A}}:\Omega_*^{ab}(\tilde{\C}[\Gamma])\to \hat{\Omega}_*^{ab}(\tilde{\mathcal{A}}(\Gamma)).$$
For an element $g\in \Gamma$, we let $<\!\!g\!\!>$ denote the associated conjugacy class. The linear span $\Omega^{<g>}_*(\C[\Gamma])$ of the sets
\begin{align}
\label{locdaefe}
\cup_{k\in \field{N}}\{g_0\mathrm{d}&g_1\mathrm{d}g_2\cdots \mathrm{d}g_k: \; g_0g_1\cdots g_k\in <\!\!g\!\!>\}\quad\mbox{and}\\ 
\nonumber
&\cup_{k\in \field{N}}\{1\cdot \mathrm{d}g_1\mathrm{d}g_2\cdots \mathrm{d}g_k: \; g_1\cdots g_k\in <\!\!g\!\!>\},
\end{align}
inside $\Omega_*^{ab}(\tilde{\C}[\Gamma])$ is clearly a sub complex; let $H^{<g>}_*(\C[\Gamma])$ denote its homology. The importance of these localized homology groups is explained in the following proposition. Its proof follows from Proposition \ref{hdrandhcyc} and \cite[Section 2.21-2.26]{Kaast}, building on the results from \cite{burgher}.

\begin{prop}
\label{hstarhstar}
There is an isomorphism 
$$b_\Gamma:H^{<e>}_*(\C[\Gamma])\xrightarrow{\sim}H_*(B\Gamma)\otimes HC_{*+2}(\C),$$
given by the composition 
\begin{align*}
H^{<e>}_*(\C[\Gamma])&\to H_*^{dR}(\C[\Gamma])\to HC_*(\C[\Gamma])\\
&\cong\qquad \begin{matrix}
H_*(B\Gamma)\otimes HC_{*}(\C)\\
\bigoplus\\ 
\bigoplus_{<\gamma>\in C'\setminus \{<e>\}}H_*(BN_\gamma)\otimes HC_*(\C)\\
\bigoplus\\ 
\bigoplus_{<\gamma>\in C''}H_*(BN_\gamma)\end{matrix}\to H_*(B\Gamma)\otimes HC_{*+2}(\C),
\end{align*}
where $C'$ denotes the conjugacy classes of finite order elements, $C''$ the conjugacy classes of infinite order elements, letting $\Gamma_\gamma$ denote the centralizer of $\gamma$ we let $N_\gamma:=\Gamma_\gamma/\gamma^\Z$, the last isomorphism is the Burghelea isomorphism (see \cite{burgher}) and the last mapping is the projection.
\end{prop}

We define the closed sub complex $\hat{\Omega}^{<g>}_*(\mathcal{A}(\Gamma))\subseteq \hat{\Omega}_*^{ab}(\tilde{\mathcal{A}}(\Gamma))$ as the closure of $j_{\Omega,\mathcal{A}}(\Omega_*^{<g>}(\C[\Gamma]))$ and denote the associated continuous inclusion by $j_{<g>}:\hat{\Omega}^{<g>}_*(\mathcal{A}(\Gamma))\to \hat{\Omega}_*^{ab}(\tilde{\mathcal{A}}(\Gamma))$. The associated homology groups will be denoted by $\hat{H}^{<g>}_*(\mathcal{A}(\Gamma))$.

\begin{define}
The relative complex of $\mathcal{A}(\Gamma)$ delocalized in $<g>$ is defined to be the homology of the mapping cone complex 
$$\hat{\Omega}_*^{rel}(\mathcal{A}(\Gamma);<g>):=\hat{\Omega}_*^{ab}(\tilde{\mathcal{A}}(\Gamma))\oplus \hat{\Omega}^{<g>}_{*-1}(\mathcal{A}(\Gamma)),$$
which is equipped with the differential
$$\mathrm{d}_{<g>}^{rel}:=
\begin{pmatrix} 
\mathrm{d}_{\hat{\Omega}_*^{ab}}& -j_{<g>}\\
0& \mathrm{d}_{\hat{\Omega}_{*-1}^{<g>}}
\end{pmatrix}.$$
We denote the homology of $\hat{\Omega}_*^{rel}(\mathcal{A}(\Gamma);<g>)$ by $\hat{H}^{del,<g>}_*(\mathcal{A}(\Gamma))$.
\end{define}

\begin{remark}
In \cite{WahlSurSet}, delocalized homology is used in the context of higher Atiyah-Patodi-Singer index theory. Beware that another definition is used; it is based on a quotient construction rather than a relative construction. 
\end{remark}

\begin{prop}
The delocalized homology groups fits into a long exact sequence on homology:
$$\ldots\to \hat{H}^{dR}_*(\mathcal{A}(\Gamma))\xrightarrow{r_{dR}} \hat{H}^{del,<g>}_*(\mathcal{A}(\Gamma))\xrightarrow{\delta_{dR}} \hat{H}^{<g>}_{*-1}(\mathcal{A}(\Gamma))\xrightarrow{\mu_{dR}} \hat{H}^{dR}_{*-1}(\mathcal{A}(\Gamma))\to \ldots$$
\end{prop}

\begin{remark}
If the induced mapping $H^{<e>}_*(\C[\Gamma])\to  \hat{H}^{<e>}_{*}(\mathcal{A}(\Gamma))$ is an isomorphism, then clearly $\hat{H}^{<e>}_{*}(\mathcal{A}(\Gamma))\cong H_*(B\Gamma)\otimes HC_{*+2}(\C)$ by Proposition \ref{hstarhstar}. This holds for instance if $\Gamma$ is a finitely generated group of polynomial growth or a hyperbolic group and $\mathcal{S}(\Gamma)\subseteq C^*_r(\Gamma)$ is the Schwartz algebra--functions with rapid decay on $\Gamma$ equipped with the convolution product, see more in \cite{connmosc, rojismo}.
\end{remark}

\begin{prop}
\label{rddel}
If $\Gamma$ is a finitely generated group of polynomial growth or hyperbolic, then $\hat{H}^{<e>}_{*}(\mathcal{S}(\Gamma))\cong H_*(B\Gamma)\otimes HC_{*+2}(\C)$ and $r_{dR}$ induces an isomorphism
$$H^{del,<e>}_*(\mathcal{S}(\Gamma)) \cong \hat{H}^{del}_*(\Gamma):=\hat{H}_*^{dR}(\mathcal{S}(\Gamma))/\hat{H}^{<e>}_{*}(\mathcal{S}(\Gamma)).$$
Furthermore, there is a dense inclusion $H^{del}_*(\Gamma)\to \hat{H}^{del}_*(\Gamma)$, where 
$$H_*^{del}(\Gamma):= \begin{matrix}
\bigoplus_{<\gamma>\in C'\setminus \{<e>\}}H_*(BN_\gamma)\otimes HC_*(\C)\\
\bigoplus\\ 
\bigoplus_{<\gamma>\in C''}H_*(BN_\gamma)\end{matrix}.$$
\end{prop}

\begin{remark}
\label{pairingdelwitcyc}
The main usage of delocalized homology will be to detect invariants on the analytic and geometric surgery groups. This is done by means of pairing the Chern character, defined in the next section, with a cyclic cocycle $\tau$ on $\mathcal{A}(\Gamma)$ satisfying $\tau(g_0,g_2,\ldots, g_k)=0$ whenever $g_0g_1\cdots g_k=e$. Assume for a moment that $\Gamma$ has polynomial growth or is hyperbolic and let $\tau_0$ be a group cocycle which has polynomial growth. As Proposition 3.8 shows, such a group cocycle $\tau_0$ will induce a cyclic cocycle of the type described above, and a mapping $\hat{H}^{del}(\Gamma) \to \C$. The choice of dense subalgebra and that of extending cocycles is in general a subtle one, addressed in for instance \cite{connmosc,gomopi,lottsuper,lotthighereta,puschholo,WahlAPSForCstarBun}.
\end{remark}

\subsection{The delocalized Chern character of a surgery cycle with connection}

We now turn to defining the delocalized Chern character for a geometric surgery cycle. As in the case of the isomorphism to the analytic surgery group, the construction is carried out by a map that cycles to cycles, which depends on a choice of geometric data. We then prove that it gives a well defined mapping from classes to classes. As before, $X$ denotes a finite CW-complex.

\begin{define}[Definitions $3.1$ and $3.13$ of \cite{paperII}]
Let $x=(W,\xi,f)$ be a cycle for $\mathcal{S}_*^{geo}(X,\mathcal{A}(\Gamma))$ with 
$$\xi=(\mathcal{E}_{\mathcal{A}(\Gamma)},\mathcal{E}'_{\mathcal{A}(\Gamma)}, E_\C,E_\C',\alpha).$$ 
A connection for $x$ is a collection $\nabla_\xi=(\nabla_W, \nabla_\mathcal{E},\nabla_{\mathcal{E}'},\nabla_ E,\nabla_{E'})$ consisting of:
\begin{itemize}
\item a spin$^c$-connection $\nabla_W$ on $W$ being of product type near $\partial W$;
\item hermitean $\mathcal{A}(\Gamma)$-linear connections $\nabla_\mathcal{E}$ and $\nabla_{\mathcal{E}'}$ on $\mathcal{E}_{\mathcal{A}(\Gamma)}$ respectively $\mathcal{E}'_{\mathcal{A}(\Gamma)}$;
\item hermitean connections $\nabla_E$ and $\nabla_{E'}$ on $E_\C$ respectively $E_\C'$.
\end{itemize}
We write a surgery cycle with connection as $(W,\xi,\nabla_\xi,f)$. 
\end{define}

We recall from \cite[Section 1.2]{WahlAPSForCstarBun} that for any manifold $W$, we obtain bicomplexes $C^\infty(W,\wedge^*T^*W\otimes \hat{\Omega}_*(\tilde{\mathcal{A}}(\Gamma)))$ and $C^\infty(W,\wedge^*T^*W\otimes \hat{\Omega}_*^{ab}(\tilde{\mathcal{A}}(\Gamma)))$ with total differential being $\mathrm{d}_W+\mathrm{d}_\mathcal{A}$. If $\mathcal{E}_{\mathcal{A}(\Gamma)}\to W$ is a smooth $\mathcal{A}(\Gamma)$-bundle, it extends via tensor product to a smooth $\tilde{\mathcal{A}}(\Gamma)$-bundle $\tilde{\mathcal{E}}_{\tilde{\mathcal{A}}(\Gamma)}:=\mathcal{E}_{\mathcal{A}(\Gamma)}\otimes_{\mathcal{A}(\Gamma)}\tilde{\mathcal{A}}(\Gamma) \to W$.
 
A connection $\nabla_\mathcal{E}$ on $\mathcal{E}_{\mathcal{A}(\Gamma)}$ can be extended to a connection $\nabla_{\tilde{\mathcal{E}}}\otimes \mathrm{d}_{\mathcal{A}}$ on the space $C^\infty(W,\wedge^*T^*W\otimes \tilde{\mathcal{E}}_{\tilde{\mathcal{A}}(\Gamma)}\otimes_{\tilde{\mathcal{A}}(\Gamma)} \hat{\Omega}_*(\tilde{\mathcal{A}}(\Gamma)))$. We recall from \cite[Lemma 1.1]{WahlAPSForCstarBun} that the element
$$\mathrm{Ch}_{\mathcal{A}(\Gamma)}(\nabla_\mathcal{E}):=\mathrm{tr}_{\mathcal{E}}\left[\mathrm{exp}(-(\nabla_{\tilde{\mathcal{E}}}\otimes \mathrm{d}_\mathcal{A})^2))\right]\in C^\infty(W,\wedge^*T^*W\otimes \hat{\Omega}_*^{ab}(\tilde{\mathcal{A}}(\Gamma)))$$
is a well defined form; it is closed with respect to the total differential $\mathrm{d}_W+\mathrm{d}_\mathcal{A}$. Given two connections $\nabla_\mathcal{E}$ and $\nabla_{\mathcal{E}}'$ on $\mathcal{E}\to W$, we let $\bar{\nabla}$ denote the connection on $\mathcal{E}\times [0,1]\to W\times [0,1]$ obtained from linearly interpolating between $\nabla_\mathcal{E}$ and $\nabla_{\mathcal{E}}'$. Furthermore, we set
$$\mathrm{Ch}_{\mathcal{A}(\Gamma)}(\nabla_\mathcal{E},\nabla_{\mathcal{E}}'):= \mathrm{Ch}_{\mathcal{A}(\Gamma)}(\bar{\nabla}).$$

\begin{define} \label{partOfChern}
Let $(W,\xi,\nabla_\xi,f)$ be a surgery cycle with connection. We can associate the following forms:
\begin{enumerate}
\item The interior Chern form of $(W,\xi,\nabla_\xi,f)$ is defined to be
$$\mathrm{Ch}^{int}(W,\xi,\nabla_\xi,f):=\mathrm{Ch}_{\mathcal{A}(\Gamma)}(\nabla_\mathcal{E})-\mathrm{Ch}_{\mathcal{A}(\Gamma)}(\nabla_{\mathcal{E}'})\in C^\infty(W,\wedge^*T^*W\otimes \hat{\Omega}_*^{ab}(\tilde{\mathcal{A}}(\Gamma))).$$ 
\item The boundary Chern form of $(W,\xi,\nabla_\xi,f)$ is defined to be
$$\mathrm{Ch}^{\partial}(W,\xi,\nabla_\xi,f):=\mathrm{Ch}_{\mathcal{A}(\Gamma)}(\nabla_E\otimes \nabla_{f^*\mathcal{L}_\mathcal{A}})-\mathrm{Ch}_{\mathcal{A}(\Gamma)}(\nabla_{E'}\otimes \nabla_{f^*\mathcal{L}_\mathcal{A}}),$$ 
as an element of $C^\infty(\partial W,\wedge^*T^*\partial W\otimes \hat{\Omega}_*^{ab}(\tilde{\mathcal{A}}(\Gamma)))$, where $\nabla_{f^*\mathcal{L}_\mathcal{A}}$ denotes the flat connection on the Mishchenko bundle $f^*\mathcal{L}_\mathcal{A}\to \partial W$. 
\item The Chern-Simons term is defined via 
$$\mathrm{CS}(W,\xi,\nabla_\xi,f):=\int_0^1\mathrm{Ch}_{\mathcal{A}(\Gamma)}(\nabla_\mathcal{E}\oplus \nabla_{E'}\otimes \nabla_{f^*\mathcal{L}},\alpha^*(\nabla_{\mathcal{E}'}\oplus \nabla_{E}\otimes \nabla_{f^*\mathcal{L}})),$$
as an element of $C^\infty(\partial W,\wedge^*T^*\partial W\otimes \hat{\Omega}_*^{ab}(\tilde{\mathcal{A}}(\Gamma)))$.
\end{enumerate}
\end{define}

The proof of the next result follows from \cite[Section IV]{lottsuper} and \cite[Section 4.3]{lotthighereta}.

\begin{prop}
\label{localorp}
If $E_\C\to M$ is a vector bundle with hermitean connections $\nabla_E$ and $\nabla_E'$, then the Chern form $\mathrm{Ch}_{\mathcal{A}(\Gamma)}(\nabla_E\otimes \nabla_{f^*\mathcal{L}})$ is in the image of the injection 
$$C^\infty(\partial W,\wedge^*T^*\partial W\otimes \hat{\Omega}_*^{<e>}(\mathcal{A}(\Gamma)))\to C^\infty(\partial W,\wedge^*T^*\partial W\otimes \hat{\Omega}_*^{ab}(\tilde{\mathcal{A}}(\Gamma))).$$
The boundary Chern form of a surgery cycle is also in the image of this injection. Moreover, the relative Chern form $\mathrm{Ch}_{\mathcal{A}(\Gamma)}(\nabla_E\otimes \nabla_{f^*\mathcal{L}},\nabla_E'\otimes \nabla_{f^*\mathcal{L}})$ is in the image of the injection 
\begin{align*}
C^\infty(\partial W\times[0,1],&\wedge^*T^*(\partial W\times[0,1])\otimes \hat{\Omega}_*^{<e>}(\mathcal{A}(\Gamma)))\\
&\to C^\infty(\partial W\times[0,1],\wedge^*T^*(\partial W\times[0,1])\otimes \hat{\Omega}_*^{ab}(\tilde{\mathcal{A}}(\Gamma)))
\end{align*}
\end{prop}

We use the notation $\mathrm{Ch}^{\partial}_{<e>}(W,\xi,\nabla_\xi,f)\in C^\infty(\partial W,\wedge^*T^*\partial W\otimes \hat{\Omega}_*^{<e>}(\mathcal{A}(\Gamma)))$ for the uniquely determined pre-image of $\mathrm{Ch}^{\partial}(W,\xi,\nabla_\xi,f)$. Similarly, we let $\mathrm{Ch}_{<e>}(\nabla_E\otimes \nabla_{f^*\mathcal{L}},\nabla_E'\otimes \nabla_{f^*\mathcal{L}})$  denote the preimage of $\mathrm{Ch}_{\mathcal{A}(\Gamma)}(\nabla_E\otimes \nabla_{f^*\mathcal{L}},\nabla_E'\otimes \nabla_{f^*\mathcal{L}})$.

\begin{lemma}
Using the notation of the previous paragraph and Definition \ref{partOfChern}, we have that
\begin{align*}
\mathrm{d}_\mathcal{A}\bigg(\int_W\mathrm{Ch}^{int}(W,\xi,\nabla_\xi,f)\wedge Td(\nabla_W)-&\int_{\partial W}\mathrm{CS}(W,\xi,\nabla_\xi,f)\wedge Td(\nabla_{\partial W})\bigg)\\
&=j_{<e>}\int_{\partial W}\mathrm{Ch}^{\partial}_{<e>}(W,\xi,\nabla_\xi,f)\wedge Td(\nabla_{\partial W}).
\end{align*}
\end{lemma}

\begin{proof}
Since $\mathrm{Ch}^{int}(W,\xi,\nabla_\xi,f)$ is $\mathrm{d}_W+\mathrm{d}_{\mathcal{A}}$-closed, it follows from Stokes's theorem that 
\begin{align*}
\mathrm{d}_\mathcal{A}\int_W\mathrm{Ch}^{int}(W,\xi,\nabla_\xi,f)\wedge Td(\nabla_W)&=\int_W\mathrm{d}_W\mathrm{Ch}^{int}(W,\xi,\nabla_\xi,f)\wedge Td(\nabla_W)\\
&=\int_{\partial W} \mathrm{Ch}^{int}(W,\xi,\nabla_\xi,f)\wedge Td(\nabla_{\partial W}).
\end{align*}
Let $\bar{\nabla}_\xi$ denote the connection on $(\mathcal{E}_\mathcal{A}|_{\partial W}\oplus E_\C'\otimes f^*\mathcal{L})\times [0,1]\to \partial W\times [0,1]$ constructed by linearly interpolating between $\nabla_\mathcal{E}\oplus\nabla_{E'}\otimes \nabla_{f^*\mathcal{L}}$ and $\alpha^*(\nabla_{\mathcal{E}'}\oplus\nabla_{E}\otimes \nabla_{f^*\mathcal{L}})$. Partially integrating over $[0,1]$, we obtain 
\begin{align*}
\int_{\partial W} \mathrm{Ch}^{int}&(W,\xi,\nabla_\xi,f)\wedge Td(\nabla_W)=\int_{\partial W}[\mathrm{Ch}_{\mathcal{A}(\Gamma)}(\nabla_\mathcal{E})-\mathrm{Ch}_{\mathcal{A}(\Gamma)}(\nabla_{\mathcal{E}'})]\wedge Td(\nabla_{\partial W})\\
=&\int_{\partial W}[\mathrm{Ch}_{\mathcal{A}(\Gamma)}(\bar{\nabla}_\xi)|_{t=1}-\mathrm{Ch}_{\mathcal{A}(\Gamma)}(\bar{\nabla}_\xi)|_{t=0}]\wedge Td(\nabla_{\partial W})\\
&+j_{<e>}\int_{\partial W}\mathrm{Ch}^{\partial}_{<e>}(W,\xi,\nabla_\xi,f)\wedge Td(\nabla_{\partial W})\\
&\qquad\qquad\qquad=j_{<e>}\int_{\partial W}\mathrm{Ch}^{\partial}_{<e>}(W,\xi,\nabla_\xi,f)\wedge Td(\nabla_{\partial W})\\
&\qquad\qquad\qquad\quad+\mathrm{d}_\mathcal{A}\int_{\partial W}\mathrm{CS}(W,\xi,\nabla_\xi,f)]\wedge Td(\nabla_{\partial W}).
\end{align*}

\end{proof}

\begin{define}
\label{delcherndefiniti}
The delocalized Chern cycle of a geometric surgery cycle with connection is the closed chain of $\hat{\Omega}_*^{rel}(\mathcal{A}(\Gamma);<e>)$ given by
\begin{align*}
\mathrm{Ch}^{del}&(W,\xi,\nabla_\xi,f)\\
&:=\begin{pmatrix}
\int_W\mathrm{Ch}^{int}(W,\xi,\nabla_\xi,f)\wedge Td(\nabla_W)-\int_{\partial W}\mathrm{CS}(W,\xi,\nabla_\xi,f)\wedge Td(\nabla_{\partial W})\\
\\
\int_{\partial W}\mathrm{Ch}^{\partial}_{<e>}(W,\xi,\nabla_\xi,f)\wedge Td(\nabla_{\partial W})\end{pmatrix}.
\end{align*}

The delocalized Chern character of a geometric surgery cycle with connection is defined as the class
$$\mathrm{ch}^{del}(W,\xi,\nabla_\xi,f):=[\mathrm{Ch}^{del}(W,\xi,\nabla_\xi,f)]\in \hat{H}_*^{del,<e>}(\mathcal{A}(\Gamma)).$$
\end{define}

\begin{lemma}
The delocalized Chern character $\mathrm{ch}^{del}(W,\xi,\nabla_\xi,f)\in \hat{H}_*^{del,<e>}(\mathcal{A}(\Gamma))$ does not depend on the choice of connection.
\end{lemma}

\begin{proof}
Let $\nabla_\xi$ and $\nabla_\xi'$ be two choices of connections on the geometric surgery cycle $(W,\xi,f)$. Then,
\begin{align*}
\mathrm{Ch}^{del}&(W,\xi,\nabla_\xi,f)-\mathrm{Ch}^{del}(W,\xi,\nabla_\xi',f)\\
=&
\begin{pmatrix}
\int_W\mathrm{Ch}^{int}(W,\xi,\nabla_\xi,f)\wedge Td(\nabla_W)-\int_{\partial W}\mathrm{CS}(W,\xi,\nabla_\xi,f)\wedge Td(\nabla_{\partial W})\\
\\
\int_{\partial W}\mathrm{Ch}^{\partial}_{<e>}(W,\xi,\nabla_\xi,f)\wedge Td(\nabla_{\partial W})
\end{pmatrix}\\
&-\begin{pmatrix}
\int_W\mathrm{Ch}^{int}(W,\xi,\nabla_\xi',f)\wedge Td(\nabla_W)-\int_{\partial W}\mathrm{CS}(W,\xi,\nabla_\xi',f)\wedge Td(\nabla_{\partial W})\\
\\
\int_{\partial W}\mathrm{Ch}^{\partial}_{<e>}(W,\xi,\nabla_\xi',f)\wedge Td(\nabla_{\partial W})
\end{pmatrix}\\
=&\,\mathrm{d}_{rel}^{<e>}
\begin{pmatrix}
0\\
\\
T_\partial(W,\xi,\nabla_\xi,\nabla_\xi',f)
\end{pmatrix},
\end{align*}
where the transgression term $T_\partial$ is defined to be
\begin{align*}
T_\partial(W,\xi,\nabla_\xi,\nabla_\xi',f)=&\int_{\partial W\times [0,1]}\mathrm{Ch}_{<e>}(\nabla_E\otimes \nabla_{f^*\mathcal{L}},\nabla_E'\otimes \nabla_{f^*\mathcal{L}})\wedge Td(\nabla_{\partial W},\nabla'_{\partial W})\\
&-\int_{\partial W\times [0,1]}\mathrm{Ch}_{<e>}(\nabla_{E'}\otimes \nabla_{f^*\mathcal{L}},\nabla_{E'}'\otimes \nabla_{f^*\mathcal{L}})\wedge Td(\nabla_{\partial W},\nabla'_{\partial W}),
\end{align*}
and we note that $Td(\nabla_{\partial W},\nabla'_{\partial W})$ denotes the Todd class associated with the connection $t\nabla_{\partial W}+(1-t)\nabla_{\partial W}'$ on $\partial W\times [0,1]$.
\end{proof}

\subsection{The delocalized Chern character from classes to classes}
\label{subsescons}

Our next goal is to study the delocalized Chern character as a mapping from classes to classes; that is, prove that it respects the relations in the group and is compatible with the other Chern characters in the geometric analogue of the surgery exact sequence. We define the localized Chern character  
$$\ch_X^{<e>}:K_*^{geo}(X)\to  \hat{H}_*^{<e>}(\mathcal{A}(\Gamma)), \quad (M,E,f)\mapsto \int_M \mathrm{Ch}_{<e>}(E\otimes f^*\mathcal{L})\wedge Td(M).$$
By standard results from Chern-Weil theory, the class of the right hand side in $\hat{H}_*^{<e>}(\mathcal{A}(\Gamma))$ is independent of choice of connection and we suppress this choice from the notation. That the map is well defined follows from Proposition \ref{localorp} (and the standard Chern-Weil theory).

\begin{theorem}
\label{delchethe}
The delocalized Chern character from Definition \ref{delcherndefiniti} induces a well defined mapping
$$\mathrm{ch}^{del}:\mathcal{S}_*^{geo}(X,\mathcal{A}(\Gamma))\to \hat{H}_*^{del,<e>}(\mathcal{A}(\Gamma)),$$
that makes the following diagram commute:
\small
\begin{equation}
\label{commdiagwithdelchern}
\begin{CD}
K^{geo}_{*}(X)@>\mu>> K_*^{geo}(pt;\mathcal{A}(\Gamma)) @>r >> \mathcal{S}^{geo}_*(X;\mathcal{A}(\Gamma))  @> \delta >> K^{geo}_{*-1}(X)  \\
@V\ch_X^{<e>}VV@V\ch_{\mathcal{A}(\Gamma)} VV @V\mathrm{ch}^{del}VV @V\ch_{X}^{<e>}VV  \\
\hat{H}_{*}^{<e>}(\mathcal{A}(\Gamma))@>\mu_{dR}>> \hat{H}^{dR}_*(\mathcal{A}(\Gamma)) @>r_{dR}>> \hat{H}_*^{del,<e>}(\mathcal{A}(\Gamma)) @>\delta_{dR}>> \hat{H}_{*-1}^{<e>}(\mathcal{A}(\Gamma))  \\
\end{CD}.
\end{equation}
\normalsize
\end{theorem}

The proof of Theorem \ref{delchethe} will be divided into several lemmas and propositions. 

\begin{lemma} 
\label{commassweld}
Assuming that $\mathrm{ch}^{del}:\mathcal{S}_*^{geo}(X,\mathcal{A}(\Gamma))\to \hat{H}_*^{del,<e>}(\mathcal{A}(\Gamma))$ is well defined, the diagram in Theorem \ref{delchethe} commutes.
\end{lemma}

\begin{proof}
The left part of the diagram commutes by construction, using Proposition \ref{localorp}. The middle part of the diagram commutes because for a geometric cycle with connection $(M,\mathcal{E}_{\mathcal{A}},\nabla_\mathcal{E},\nabla_M)$ for $K_*^{geo}(pt;\mathcal{A}(\Gamma))$, it holds on the level of cycles that
$$r_{dR}\left(\mathrm{Ch}(M,\mathcal{E}_\mathcal{A},\nabla_\mathcal{E},\nabla_M)\right)=
\begin{pmatrix}
\int_M \mathrm{Ch}(\nabla_\mathcal{E})\wedge Td(\nabla_M)\\ 0\end{pmatrix}=\mathrm{Ch}^{del}(r(M,\mathcal{E}_{\mathcal{A}},\nabla_\mathcal{E},\nabla_M)).$$
Finally, the right part of the diagram also commutes on the level of cycles, because for a geometric surgery cycle with connection $(W,\xi,\nabla_\xi,f)$ it holds that
\begin{align*}
\delta_{dR}(\ch^{del}_X(W,\xi,f))&= \int_{\partial W}\mathrm{Ch}^{\partial}_{<e>}(W,\xi,\nabla_\xi,f)\wedge Td(\nabla_{\partial W})\\
&=\ch^{<e>}_{X}\left((\partial W,E,f)\dot{\cup}(-\partial W,E',f)\right)=\ch^{<e>}_{X}(\delta(W,\xi,f)).
\end{align*}
\end{proof}

\begin{lemma}
\label{vanonimmu}
Assume that $(W_0,\xi_0,f_0)$ is a cycle as in Lemma \ref{delapsforcycles}. Let $\nabla_{\xi_0}=(\nabla_W, \nabla_\mathcal{E},\nabla_{\mathcal{E}'},\nabla_ E,\nabla_{E'})$ be a connection for $\xi_0$, $\nabla_F$ a connection on $F_\C$ extending $\nabla_E$ and $\nabla_{F'}$ a connection on $F'_\C$ extending $\nabla_{E'}$. In the notation of Lemma \ref{delapsforcycles}, we have that
\begin{align*}
\mathrm{Ch}^{del}&(W,\xi,\nabla_\xi,f)\\
&=\mathrm{d}_{<e>}^{rel}\begin{pmatrix}
\int_{W_0\times [0,1]} \mathrm{Ch}_{\mathcal{A}(\Gamma)}\left(\nabla_{\mathcal{E}}\oplus \nabla_{F'}\otimes \nabla_{g^*\mathcal{L}},\beta^*(\nabla_{\mathcal{E}'}\oplus \nabla_{F}\otimes \nabla_{g^*\mathcal{L}})\right)\wedge Td(\nabla_{W_0}) \\
\\
\int_{W_0}\big(\mathrm{Ch}_{<e>}(\nabla_F\otimes \nabla_{g^*\mathcal{L}_\mathcal{A}})-\mathrm{Ch}_{<e>}(\nabla_{F'}\otimes \nabla_{g^*\mathcal{L}_\mathcal{A}})\big)\wedge Td(\nabla_{W_0})\end{pmatrix}
.
\end{align*}
Hence $\mathrm{ch}^{del}(W,\xi,\nabla_\xi,f)=0$.
\end{lemma}

\begin{proof}
To simplify notation, we set 
$$z:=\int_{W_0\times [0,1]} \mathrm{Ch}_{\mathcal{A}(\Gamma)}\left(\nabla_{\mathcal{E}}\oplus \nabla_{F'}\otimes \nabla_{g^*\mathcal{L}},\beta^*(\nabla_{\mathcal{E}'}\oplus \nabla_{F}\otimes \nabla_{g^*\mathcal{L}})\right)\wedge Td(\nabla_{W_0}).$$
We compute that 
\begin{align*}
\mathrm{d}_{<e>}^{rel}&\begin{pmatrix}
z \\
\\
\int_{W_0}\big(\mathrm{Ch}_{<e>}(\nabla_F\otimes \nabla_{g^*\mathcal{L}_\mathcal{A}})-\mathrm{Ch}_{<e>}(\nabla_{F'}\otimes \nabla_{g^*\mathcal{L}_\mathcal{A}})\big)\wedge Td(\nabla_{W_0})\end{pmatrix}\\
\\
&=\begin{pmatrix}
\mathrm{d}_\mathcal{A} z-\int_{W_0} \left(\mathrm{Ch}_{\mathcal{A}(\Gamma)}(\nabla_{F}\otimes \nabla_{g^*\mathcal{L}})-\mathrm{Ch}_{\mathcal{A}(\Gamma)}( \nabla_{F'}\otimes \nabla_{g^*\mathcal{L}})\right)\wedge Td(\nabla_{W_0}) \\
\\
\int_{\partial W_0}\big(\mathrm{Ch}_{<e>}(\nabla_E\otimes \nabla_{f_0^*\mathcal{L}_\mathcal{A}})-\mathrm{Ch}_{<e>}(\nabla_{E'}\otimes \nabla_{f_0^*\mathcal{L}_\mathcal{A}})\big)\wedge Td(\nabla_{\partial W_0})\end{pmatrix}\\
\\
&=\begin{pmatrix}
\int_{ W_0}(\mathrm{Ch}_{\mathcal{A}(\Gamma)}(\nabla_\mathcal{E})-\mathrm{Ch}_{\mathcal{A}(\Gamma)}(\nabla_{\mathcal{E}'}))\wedge  Td(\nabla_{W_0})\\
-\int_{\partial W_0}\big(\mathrm{Ch}_{\mathcal{A}(\Gamma)}(\nabla_E\otimes \nabla_{f_0^*\mathcal{L}_\mathcal{A}})-\mathrm{Ch}_{\mathcal{A}(\Gamma)}(\nabla_{E'}\otimes \nabla_{f_0^*\mathcal{L}_\mathcal{A}})\big)\wedge Td(\nabla_{\partial W_0})\\
\\
\int_{\partial W_0}\big(\mathrm{Ch}_{<e>}(\nabla_E\otimes \nabla_{f_0^*\mathcal{L}_\mathcal{A}})-\mathrm{Ch}_{<e>}(\nabla_{E'}\otimes \nabla_{f_0^*\mathcal{L}_\mathcal{A}})\big)\wedge Td(\nabla_{\partial W_0})\end{pmatrix}\\
\\
&=\begin{pmatrix}
\int_{W_0}\mathrm{Ch}^{int}(W_0,\xi,\nabla_\xi,f)\wedge Td(\nabla_{W_0})-\int_{\partial W_0}\mathrm{CS}(W,\xi,\nabla_\xi,f)\wedge Td(\nabla_{\partial W_0})\\
\\
\int_{\partial W_0}\mathrm{Ch}^{\partial}_{<e>}(W,\xi,\nabla_\xi,f)\wedge Td(\nabla_{\partial W_0})\end{pmatrix}.
\end{align*}
\end{proof}

\begin{lemma}
\label{bordindel}
The delocalized Chern character respects the bordism relation in $\mathcal{S}_*^{geo}(X;\mathcal{A}(\Gamma))$.
\end{lemma}

\begin{proof}
Suppose that $(W,\xi,f)\sim_{bor}0$ in $\mathcal{S}_*^{geo}(X;\mathcal{A}(\Gamma))$. Hence, there is a cycle $x_0=(W_0,\xi_0,f)$ as in Lemma \ref{delapsforcycles} and Lemma \ref{vanonimmu} such that $\partial W=\partial W_0$ and $W\cup_{\partial W} W_0$ being the boundary of a spin$^c$-manifold $Z$, with bundle data such that $\mathcal{E}\cup_{\partial W} \mathcal{E}_0\to \partial Z$ extends to a bundle $\mathcal{F}\to Z$ and $\mathcal{E}'\cup_{\partial W} \mathcal{E}_0'\to \partial Z$ extends to a bundle $\mathcal{F}'\to Z$. Choose connections $\nabla_F$, $\nabla_{F'}$, $\nabla_{\mathcal{F}}$, $\nabla_{\mathcal{F}'}$ and $\nabla_Z$ of product type near the relevant boundaries and use these connections to define a connection $\nabla_\xi$ on $(W,\xi,f)$. 

It follows that 
\begin{align*}
\mathrm{Ch}^{del}&(W,\xi,\nabla_\xi,f)\\
&=\begin{pmatrix}
\int_{\partial Z}(\mathrm{Ch}_{\mathcal{A}(\Gamma)}(\nabla_\mathcal{F})-\mathrm{Ch}_{\mathcal{A}(\Gamma)}(\nabla_{\mathcal{F}'}))\wedge  Td(\nabla_{\partial Z})\\
+\int_{W_0}\big(\mathrm{Ch}_{\mathcal{A}(\Gamma)}(\nabla_{\mathcal{E}_0})-\mathrm{Ch}_{\mathcal{A}(\Gamma)}(\nabla_{\mathcal{E}_0'})\big)\wedge Td(\nabla_{W_0})\\
-\int_{\partial W} CS(W,\xi,\nabla_\xi,f)\wedge Td(\nabla_{\partial W})\\
\\
\int_{\partial W}\big(\mathrm{Ch}_{<e>}(\nabla_E\otimes \nabla_{f^*\mathcal{L}_\mathcal{A}})-\mathrm{Ch}_{<e>}(\nabla_{E'}\otimes \nabla_{f^*\mathcal{L}_\mathcal{A}})\big)\wedge Td(\nabla_{\partial W})\end{pmatrix}\\
\\
&=\begin{pmatrix}
\int_{\partial Z}(\mathrm{Ch}_{\mathcal{A}(\Gamma)}(\nabla_\mathcal{F})-\mathrm{Ch}_{\mathcal{A}(\Gamma)}(\nabla_{\mathcal{F}'}))\wedge  Td(\nabla_{\partial Z})\\
\\
0\end{pmatrix}+\mathrm{Ch}^{del}(W_0,\xi_0,\nabla_{\xi_0},f_0).
\end{align*}
Partial integration implies that 
\begin{align*}
&\begin{pmatrix}
\int_{\partial Z}(\mathrm{Ch}_{\mathcal{A}(\Gamma)}(\nabla_\mathcal{F})-\mathrm{Ch}_{\mathcal{A}(\Gamma)}(\nabla_{\mathcal{F}'}))\wedge  Td(\nabla_{\partial Z})\\
\\
0\end{pmatrix}\\
&\qquad =\mathrm{d}^{rel}_{<e>}\begin{pmatrix}
\int_{Z}(\mathrm{Ch}_{\mathcal{A}(\Gamma)}(\nabla_\mathcal{F})-\mathrm{Ch}_{\mathcal{A}(\Gamma)}(\nabla_{\mathcal{F}'}))\wedge  Td(\nabla_Z)\\
\\
0\end{pmatrix}.
\end{align*}
This identity, combined with Lemma \ref{vanonimmu}, guarantees that $\mathrm{Ch}^{del}(W,\xi,\nabla_\xi,f)$ is exact.
\end{proof}

For a spin$^c$-vector bundle $V\to W$ of rank $2k$, the Bott bundle $Q_V\to W^V=S(V\oplus 1_\field{R})$ carries a hermitean connection $\nabla_Q$ coming from a spin$^c$-invariant connection on the Bott bundle $Q_k\to S^{2k}$. Similarly, the vertical space of $W^V\to W$ carries a connection $\nabla_{W^V/W}$.

\begin{define}[Vector bundle modification of connections]
Let $(W,\xi,\nabla_\xi,f)$ be a cycle with connection $\nabla_\xi=(\nabla_W, \nabla_\mathcal{E},\nabla_{\mathcal{E}'},\nabla_ E,\nabla_{E'})$ and $V\to W$ a spin$^c$-vector bundle of even rank. We define the vector bundle modified connection $\nabla_\xi^V$ by
$$\nabla_\xi^V:=(\pi^*\nabla_W\otimes \nabla_{W^V/W}, \pi^*\nabla_\mathcal{E}\otimes \nabla_Q,\pi^*\nabla_{\mathcal{E}'}\otimes \nabla_Q,\pi^*\nabla_ E\otimes \nabla_Q,\pi^*\nabla_{E'}\otimes \nabla_Q).$$
The vector bundle modification of the cycle with connection $(W,\xi,\nabla_\xi,f)$ is defined by
$$(W,\xi,\nabla_\xi,f)^V:=(W^V,\xi^V,\nabla_\xi^V,f^V).$$
\end{define}

The following proposition follows directly from the proof of \cite[Theorem 1.26]{Kaast}.

\begin{prop}
\label{vebmodinv}
Let $(W,\xi,\nabla_\xi,f)$ be a geometric cycle for surgery with connection and $V\to W$ a spin$^c$-vector bundle of even rank. Then 
\begin{align*}
\int_{W^V}\mathrm{Ch}^{int}(W,\xi,\nabla_\xi,f)^V\wedge   Td(\nabla_{W}^V)&=\int_W\mathrm{Ch}^{int}(W,\xi,\nabla_\xi,f)\wedge  Td(\nabla_{ W});\\
\int_{\partial W^V}\mathrm{CS}_{<e>}(W,\xi,\nabla_\xi,f)^V\wedge  Td(\nabla_{\partial W}^V)&=\int_{\partial W}\mathrm{CS}_{<e>}(W,\xi,\nabla_\xi,f)\wedge  Td(\nabla_{\partial W})\\
\int_{\partial W^V}\mathrm{Ch}^{\partial}_{<e>}(W,\xi,\nabla_\xi,f)^V \wedge  Td(\nabla_{\partial W}^V)&=\int_{\partial W}\mathrm{Ch}^{\partial}_{<e>}(W,\xi,\nabla_\xi,f)\wedge  Td(\nabla_{\partial W}).
\end{align*}
In particular, the delocalized Chern character respects vector bundle modification in $\mathcal{S}_*^{geo}(X;\mathcal{A}(\Gamma))$.
\end{prop}

The proof of Theorem \ref{delchethe} is obtained by combining Lemma \ref{commassweld}, Lemma \ref{bordindel} and Proposition \ref{vebmodinv}.

\subsection{Smooth sub-algebras and the delocalized Chern character}
\label{smootthsubs}

Suppose that $\mathcal{A}(\Gamma)\subseteq C^*_r(\Gamma)$ is a Fr\'echet  $*$-sub algebra. Following \cite{paperI}, the geometric surgery group $ \mathcal{S}^{geo}_{*}(B\Gamma;\mathcal{A}(\Gamma))$ can be defined and it fits into a geometric  surgery exact sequence with $K^{geo}_{*}(B\Gamma)$ and $K_{*}^{geo}(pt;\mathcal{A}(\Gamma))$. By naturality, there is a commuting diagram with exact rows:
\tiny
\[\begin{CD}
\ldots  @> \delta >> K^{geo}_{*}(B\Gamma) @>\mu>>K_{*}^{geo}(pt;\mathcal{A}(\Gamma)) @>r >> \mathcal{S}^{geo}_{*}(B\Gamma;\mathcal{A}(\Gamma))  @> \delta >>K^{geo}_{*-1}(B\Gamma) @>>> \ldots\\
@. @V=VV @VVV @VVV@V=VV@. \\
\ldots  @> \delta >> K^{geo}_{*}(B\Gamma) @>\mu>>K_{*}^{geo}(pt;C^*_r(\Gamma)) @>r >> \mathcal{S}^{geo}_{*}(B\Gamma;C^*_r(\Gamma))  @> \delta >>K^{geo}_{*-1}(B\Gamma) @>>> \ldots
\end{CD},\]
\normalsize
where the vertical arrows are induced by the inclusion $\mathcal{A}(\Gamma)\subseteq C^*_r(\Gamma)$. If the induced mapping $K_*^{geo}(pt;\mathcal{A}(\Gamma))\to K_*^{geo}(pt;C^*_r(\Gamma))$ is an isomorphism; we call such a subalgebra a $K$-smooth algebra. Examples of $K$-smooth algebra include dense smooth subalgebras and Banach algebra completions $\mathcal{A}(\Gamma)$ such that $K_*(\mathcal{A}(\Gamma))\to K_*(C^*_r(\Gamma))$ is an isomorphism, see \cite[Appendix]{paperI}. It follows from the five Lemma that for a $K$-smooth  subalgebra $\mathcal{A}(\Gamma)\subseteq C^*_r(\Gamma)$, the inclusion mapping induces an isomorphism $\mathcal{S}^{geo}_{*}(B\Gamma;\mathcal{A}(\Gamma))\cong \mathcal{S}^{geo}_{*}(B\Gamma;C^*_r(\Gamma))$; which after composing with $\lambda_{an}^{-1}$ produces a delocalized Chern character on the analytic surgery group. \\

Assume that $\Gamma$ is a group of polynomial growth or hyperbolic. We consider the Schwartz algebra $\mathcal{S}(\Gamma)$ which since $\Gamma$ has property $(RD)$ (see \cite{Haa} for more on Property $(RD)$) is a dense smooth subalgebra of $C^*_r(\Gamma)$. By the argument above, there is a delocalized Chern character 
$$\widetilde{\ch}^{del}:\mathcal{S}^{geo}_{*}(B\Gamma;C^*_r(\Gamma))\to \hat{H}^{del}_*(\Gamma).$$
The reader is directed to  Proposition \ref{rddel} for notation. 

Recalling the notations of Subsection \ref{secttwoone},  we let $(W,\Xi,f)$ be a decorated cycle for $\mathcal{S}_*^{geo}(X;\mathcal{S}(\Gamma))$. We can assume that the choice of Dirac operators contained in $\Xi$ is defined from a choice of connections $\nabla_\xi$ on $\xi$ and also that all the trivializing operators are induced from trivializing operators with coefficients in $\mathcal{S}(\Gamma)$, see \cite[Proposition 2.10]{LPGAFA}. Following \cite{lotthighereta, PS} we define the delocalized $\rho$-invariant of $(W,\Xi,f)$ by 
\begin{align*}
\hat{\rho}(W,\Xi,f):&=\hat{\rho}(D_{E\otimes f^*\mathcal{L}}\dot{\cup}D_{E'\otimes f^*\mathcal{L}}+A)\\
&=\hat{\rho}(D_\mathcal{E}^{\partial W}+A^\mathcal{E})+\hat{\rho}(-D_{\mathcal{E}'}^{\partial W}-A^{\mathcal{E}'})+\hat{\rho}(\bar{D}_\Xi^\partial+A^\Xi)\in \hat{H}^{del}_*(\Gamma).
\end{align*}
Here $\hat{\rho}$ denotes the delocalized $\rho$-invariant of a Dirac operator perturbed by a smoothing operator, see more in \cite[Section 4.6]{lotthighereta} and \cite{WahlAPSForCstarBun}. It is the part of the higher $\eta$-invariant belonging to the range of the projection mapping 
$$\pi_{del}:\hat{H}_*^{dR}(\mathcal{S}(\Gamma))\cong H_*(B\Gamma)\otimes HC_{*+2}(\C)\oplus \hat{H}^{del}_*(\Gamma)\to \hat{H}^{del}_*(\Gamma).$$ 
Even though the higher $\eta$-invariant is not closed, its delocalized part is, giving a well defined delocalized $\rho$-invariant. It follows from \cite{LPGAFA}, that $\ind_{APS}(W,\Xi,f)$ can be constructed as an element of $K_*(\mathcal{S}(\Gamma))$ and as such has a Chern character $\ch_{\mathcal{S}(\Gamma)}\left(\ind_{APS}(W,\Xi,f)\right)\in \hat{H}_*^{dR}(\mathcal{S}(\Gamma))$. The following Proposition is an immediate consequence of Wahl's higher Atiyah-Patodi-Singer index theorem \cite[Theorem 9.4]{WahlAPSForCstarBun}.

\begin{prop}
Let $\Gamma$ be a discrete group of polynomial growth or hyperbolic. Then $\hat{\rho}(W,\Xi,f)$ satisfies that 
$$\pi_{del}\left(\ch_{\mathcal{S}(\Gamma)}\left(\ind_{APS}(W,\Xi,f)\right)\right)-\hat{\rho}(W,\Xi,f)=\widetilde{\ch}^{del}(W,\xi,f).$$
\end{prop}

We return to the case of a general group $\Gamma$ and consider two unitary representations $\sigma_1,\sigma_2:\Gamma\to U(k)$. Assume that $\sigma_1$ and $\sigma_2$ extend continuously to a Fr\'echet  algebra completion $\mathcal{A}(\Gamma)$ of $\C[\Gamma]$. Two such representations give rise to a relative $\eta$-invariant $\rho_{\sigma_1,\sigma_2}:\mathcal{S}_0^{geo}(B\Gamma;\mathcal{A}(\Gamma))\to \field{R}$, see more in \cite{paperII}. This relative $\eta$-invariant factorizes over the construction of Higson-Roe \cite{HReta}, see \cite[Section 3]{paperII}. On the other hand, $\sigma_1$ and $\sigma_2$ define a continuous cyclic $0$-cocycle $\tau_{\sigma_1,\sigma_2}$ on $\mathcal{A}(\Gamma)$ such that $\tau_{\sigma_1,\sigma_2}(1)=0$. By Remark \ref{pairingdelwitcyc}, we obtain a pairing with delocalized de Rham homology $\hat{H}_{2k}^{del,<e>}(\mathcal{A})\ni x\mapsto \langle \tau_{\sigma_1,\sigma_2},x\rangle\in \C$. 

The mapping $\langle \tau_{\sigma_1,\sigma_2},\cdot \rangle:\hat{H}_{2k}^{del,<e>}(\mathcal{A})\to \C$ can be more explicitly described since $\sigma_1$ and $\sigma_2$ are representations: each $\sigma_i$ induces a mapping $(\sigma_i)_*:\hat{\Omega}_*^{ab}(\mathcal{A})\to \hat{\Omega}_*^{ab}(M_k(\C))$. Using Proposition \ref{hdrandhcyc}, we have isomorphisms 
$$\hat{H}^{dR}_{2k}(M_k(\C))=H^{dR}_{2k}(M_k(\C))\cong SHC_{2k+2}(M_k(\C))\cong SHC_{2k+2}(\C)\cong \C.$$ 
We denote the isomorphism $\hat{H}^{dR}_{2k}(M_k(\C))\cong \C$ by $\omega$. Using the isomorphism $\omega$, it is immediate from definition in Equation \eqref{locdaefe} that the two induced mappings $\omega\circ (\sigma_i)_*:\hat{H}_{2k}^{<e>}(\mathcal{A})\to \C$ are equal. Therefore, we can for a relative cycle $x=(x_1,x_2)\in \hat{\Omega}_{0}^{rel}(\mathcal{A}(\Gamma);<e>)$ define the pairing with $\tau_{\sigma_1,\sigma_2}$ by
$$\langle \tau_{\sigma_1,\sigma_2},x\rangle:=\omega\circ \left((\sigma_1)_*(x_1)-(\sigma_2)_*(x_1)\right).$$

\begin{prop}
\label{numericalinv}
Let $\mathcal{A}(\Gamma)$ be a Fr\'echet  algebra completion of $\C[\Gamma]$ and $\sigma_1,\sigma_2:\Gamma\to U(k)$  two unitary representations continuous in the topology of $\mathcal{A}(\Gamma)$. The following identity holds for any $w\in \mathcal{S}_0^{geo}(B\Gamma;\mathcal{A}(\Gamma))$:
$$\rho_{\sigma_1,\sigma_2}(w)=\langle \tau_{\sigma_1,\sigma_2},\ch^{del}(w)\rangle.$$
\end{prop}

\begin{proof}
Let $(W,\xi,\nabla_\xi,f)$ be a cycle for $\mathcal{S}_0^{geo}(B\Gamma;\mathcal{A}(\Gamma))$ representing $w$. By the discussion above, 
\begin{align*}
\langle \tau_{\sigma_1,\sigma_2},\ch^{del}(w)\rangle=\omega\circ (\sigma_1-\sigma_2)_*\bigg(&\int_W\mathrm{Ch}^{int}(W,\xi,\nabla_\xi,f)\wedge Td(\nabla_W)\\
&-\int_{\partial W}\mathrm{CS}(W,\xi,\nabla_\xi,f)\wedge Td(\nabla_{\partial W})\bigg).
\end{align*}
Following the notation in \cite[Section 3]{paperII}, we define the vector bundles 
$$E_i:=\mathcal{L}_{\mathcal{A}(\Gamma)}\otimes_{\sigma_i} \C^k,\quad G_i:= \mathcal{E}_{\mathcal{A}(\Gamma)}\otimes_{\sigma_i}\C^k \quad\mbox{and}\quad G_i':=\mathcal{E}'_{\mathcal{A}(\Gamma)}\otimes_{\sigma_i}\C^k.$$
We let $\nabla_{G_i}$ denote the connection constructed from $\nabla_{\mathcal{E}}$ and $\nabla_{G_i'}$ the connection constructed from $\nabla_{\mathcal{E}'}$. Similarly, $\nabla_{E,i}$ and $\nabla_{E',i}$ are constructed from twisting $\nabla_{E}$ and $\nabla_{E'}$, respectively, by the flat connection on $f^*E_i$. We define $\bar{\nabla}_i$ as the connection on $(G_i\oplus E'_{\C}\otimes f^*E_i)\times [0,1]\to \partial W\times [0,1]$ obtained from interpolating between $\nabla_{G_i}\oplus \nabla_{E',i}$ and $\alpha^*(\nabla_{G_i'}\oplus \nabla_{E,i})$. In terms of the connections $(\nabla_{G_1},\nabla_{G_2}, \nabla_{G_1'},\nabla_{G_2'}, \bar{\nabla}_1,\bar{\nabla}_2)$, we can write $\langle \tau_{\sigma_1,\sigma_2},\ch^{del}(w)\rangle$ as
\begin{align*}
\omega\circ (\sigma_1&-\sigma_2)_*\bigg(\int_W\mathrm{Ch}^{int}(W,\xi,\nabla_\xi,f)\wedge Td(\nabla_W)\\
&\qquad\qquad\qquad-\int_{\partial W}\mathrm{CS}(W,\xi,\nabla_\xi,f)\wedge Td(\nabla_{\partial W})\bigg)\\
=\int_W \bigg(\ch(&\nabla_{G_1})-\ch(\nabla_{G_2})\bigg)\wedge Td(\nabla_W)-\int_W \bigg(\ch(\nabla_{G'_1})-\ch(\nabla_{G'_2})\bigg)\wedge Td(\nabla_W)\\
&+\int_{\partial W\times [0,1]} \left(\ch(\bar{\nabla}_1)-\ch(\bar{\nabla}_2)\right)Td(\nabla_{\partial W})=\rho_{\sigma_1,\sigma_2}(w),
\end{align*}
where we in the last equality use the definition of $\rho_{\sigma_1,\sigma_2}$ (see \cite[Definitions 3.1 and 3.2]{paperII}, and also \cite[Remark 3.3]{paperII}).
\end{proof}

\begin{remark}
It would be interesting to extend the content of Proposition \ref{numericalinv} to higher cocycles. The mere formulation for higher cocycles is analytically more subtle, see Remark \ref{pairingdelwitcyc}. In fact, already the problem of extending unitary representations is analytically subtle. If a finite dimensional unitary representation $\sigma$ extends to a dense holomorphically closed $*$-subalgebra $\mathcal{A}(\Gamma)\subseteq C^*_r(\Gamma)$, then $\sigma$ automatically extends to $C^*_r(\Gamma)$. In particular, if $\Gamma$ has property (T), this causes substantial difficulties for proving rigidity results for relative $\eta$-invariants, see \cite[Section 5]{paperII}. 
\end{remark}

\section{Mappings from Stolz' positive scalar curvature exact sequence}
\label{mapstotosu}

The positive scalar curvature exact sequence of Stolz describes the relationship between positive scalar curvature metrics on spin manifolds with fundamental group $\Gamma$ and the spin bordism class of the manifolds. For details regarding the background to this section, we refer to \cite{PSrhoInd,SchPSCsur,stolzicm}. 

Let $\Gamma$ be a finitely generated discrete group and $X$ be a finite CW-complex which is also a closed subspace of $B\Gamma$; also let $n\in \field{N}$. We write $\Omega_n^{Spin}(X)$ for the spin bordism group of $X$. That is, $\Omega_n^{Spin}(X)$ consists of equivalence classes of pairs $(M,f:M\to X)$, where $M$ is a closed $n$-dimensional spin-manifold (where the equivalence relation is spin-bordism); addition is defined using disjoint union.

The bordism group $\mathrm{Pos}_n^{Spin}(X)$ of spin manifolds over $X$ with a positive scalar curvature metric is defined as follows: a cycle for $\mathrm{Pos}_n^{Spin}(X)$ is a triple $(M,f:M\to X,g)$ consisting of a cycle $(M,f)$ for $\Omega_n^{Spin}(X)$ and a metric $g$ on $M$ of positive scalar curvature. If $(W,f:W\to X,g)$ is as in a cycle for $\mathrm{Pos}_n^{Spin}(X)$, but $W$ has boundary, then we say that $(\partial W,f|_{\partial W},g|_{\partial W})\sim_{bor} 0$. The equivalence classes of cycles under the associated bordism relation leads to the abelian group, $\mathrm{Pos}_n^{Spin}(X)$; again, addition is defined via disjoint union. 

The forgetful mapping $(M,f,g)\mapsto (M,f)$ gives a well defined morphism of abelian groups
$$\mathrm{Pos}_n^{Spin}(X)\to \Omega_n^{Spin}(X).$$
This mapping is of relevance for positive scalar curvature metrics because of results of Gromov-Lawson-Rosenberg implying that a closed spin manifold $M$ of dimension $n\geq 5$, with classifying mapping $f:M\to B\pi_1(M)$, admits a positive scalar curvature metric if and only if $[M,f]\in \mathrm{im} (\mathrm{Pos}_n^{Spin}(B\pi_1(M))\to \Omega_n^{Spin}(B\pi_1(M)))$, see more in for instance \cite[Theorem 3.5]{stolzicm}.

The final ingredient in Stolz' positive scalar curvature exact sequence is a relative group $R_n^{Spin}(X)$ constructed from the mapping $\mathrm{Pos}_n^{Spin}(X)\to \Omega_n^{Spin}(X)$. A cycle for $R_n^{Spin}(X)$ is a triple $(W,f:W\to X,g)$ where $W$ is an $n$-dimensional spin manifold with boundary and $g$ is a metric on $\partial W$ with positive scalar curvature. The bordism relation in $R_n^{Spin}(X)$ is constructed similarly to any relative group: a cycle for $R_n^{Spin}(X)$ with boundary is a triple $((Z,W), f:Z\to X, g)$ where $Z$ is an $n+1$-dimensional spin manifold with boundary, $W\subseteq \partial Z$ is a regular compact domain and $g$ is a metric on $\partial Z\setminus W^\circ$ with positive scalar curvature. The bordism relation is defined in the usual manner; in particular, the reader should note that
$$\partial ((Z,W), f:Z\to X, g):=(W, f|_W, g_{\partial W})$$ 

There are by construction well defined morphisms 
$$\Omega_n^{Spin}(X)\to R_n^{Spin}(X), \quad (M,f)\mapsto (M,f, \emptyset),\quad\mbox{and}$$
$$R_{n+1}^{Spin}(X)\to \mathrm{Pos}_n^{Spin}(X), \quad (W,f,g)\mapsto (\partial W,f|_{\partial W},g).$$
The positive scalar curvature exact sequence of Stolz is the long exact sequence 
$$\ldots\to \mathrm{Pos}_n^{Spin}(X)\to \Omega_n^{Spin}(X)\to R_n^{Spin}(X)\to\mathrm{Pos}_{n-1}^{Spin}(X)\to \ldots$$
\\

We now turn to mapping the positive scalar curvature sequence into the geometric analog of the analytic surgery exact sequence. Using the notations of \cite{PSrhoInd}, we define 
$$\beta^{geo}: \Omega_n^{Spin}(X)\to K_n^{geo}(X), \quad (M,f)\mapsto (M,M\times \C,f).$$
The mapping $\beta^{geo}$ is clearly well defined, because it maps spin bordisms to bordisms in $K_n^{geo}(X)$. In order to define the mappings from $R_{n}^{Spin}(X)$ respectively $\mathrm{Pos}_n^{Spin}(X)$, we need the following Proposition.

\begin{prop}
\label{decorpsc}
Let $(W,\xi,h)$ be a cycle for $\mathcal{S}^{geo}_*(X;\mathcal{L}_X)$. Assume that 
\begin{itemize}
\item The boundary  takes the form $\partial W=M^V$ for a spin manifold $M$ and some spin$^c$-vector bundle $V\to M$ of rank $2k$;
\item the mapping $h$ is of the form $h=f^V$ for a continuous $h:M\to X$;
\item the relative $K$-theory cocycle has the form
$$\xi=(\mathcal{E}_{C^*_r(\Gamma)},\mathcal{E}'_{C^*_r(\Gamma)},Q_V, \C^{2^{k}}, \alpha)$$ 
where $Q_V\to M^V$ denotes the Bott bundle and $\mathcal{E}_{C^*_r(\Gamma)},\mathcal{E}'_{C^*_r(\Gamma)}\to W$ are $C^*_r(\Gamma)$-bundles.
\end{itemize}
Then given a positive scalar curvature metric $g$ on $M$ with associated spin Dirac operator $D_g$, there is a decoration of $\xi$ that takes the form
$$\Xi_g=(\xi,(D_\mathcal{E},D_{\mathcal{E}'},D_g^V,D^0),(A^\mathcal{E},A^{\mathcal{E}'},0\dot{\cup} A^0)).$$
After constructing a Dirac operator $D^B$ on the ball bundle $M^B:=B(V\oplus 1_\field{R})\to M$ with boundary operator $D^0$, the class
\begin{align*}
\nu(W,\xi,f,g):=&j_X\left(\ind_{APS}(W,\Xi_g,f)\right)\\ 
&-j_X\left(\ind_{APS}(M^B,((\mathcal{L}_{M^B}^{2^k},\C^{2^k},\mathrm{id}),(D^{B},D^0),A^0),f)\right)\in K_*(D^*(\tilde{X})^\Gamma),
\end{align*}
does not depend on the choice of $\Xi_g$ but only on the class $[W,\xi,f]$ and the metric $g$ on $M$.
\end{prop}

\begin{proof}
Existence of decorations of the specified form follows directly from the invertibility of spin Dirac operators defined from positive scalar curvature metrics and the techniques of Subsection \ref{vbintermezz}. By the delocalized APS-theorem (see Theorem \ref{delapsthemps}) 
\begin{equation}
\label{delapsforball}
j_X\ind_{APS}(M^B,((\mathcal{L}_{M^B}^{2^k},\C^{2^k},\mathrm{id}),(D^{B},D^0),A^0),f)=\rho(D^0,A^0).
\end{equation}
Equation \eqref{delapsforball} and the construction of $\lambda_{an}^\mathcal{S}$ imply that
\begin{align*}
\lambda_{an}^\mathcal{S}(W,\xi,f)-&\rho(D_g^V,0)=j_X\left(\ind_{APS}(W,\Xi_g,f)\right)\\ 
&\quad-j_X\left(\ind_{APS}(M^B,((\mathcal{L}_{M^B}^{2^k},\C^{2^k},\mathrm{id}),(D^{B},D^0),A^0),f)\right).
\end{align*}
By Lemma \ref{rhoandvb}, $\rho(D_g^V,0)=\rho(D_g,0)$ (which clearly only depend on $g$). The result now follows since $\lambda_{an}^\mathcal{S}(W,\xi,f)$ only depends on the class $[W,\xi,f]$ (see Theorem \ref{themaplambda}).
\end{proof}

Our next goal is the definition of a map
$$\ind_{\Gamma}^{geo}:R_n^{Spin}(X)\to K_n^{geo}(pt;C^*_r(\Gamma)).$$
To do so, further notation is required. Let $(W,f,g)$ be a cycle for $R_n^{Spin}(B\Gamma)$. We define the relative $K$-theory cocycle for assembly $\xi_f:=(f^*\mathcal{L}_{X},\C,\mathrm{id}_{f^*\mathcal{L}_{X}})$ on $(W,\partial W)$. Since $g$ is a positive scalar curvature metric on $\partial W$, we can find a decoration $\Xi_{f,g}$ of $\xi_f$ as in Proposition  \ref{decorpsc}. We define 
$$\ind_{\Gamma}^{geo}:(W,f,g)\mapsto (pt,\ind_{APS}(W,\Xi_{f,g},f)).$$

\begin{prop}
The map $\ind_{\Gamma}^{geo}:R_n^{Spin}(X)\to K_n^{geo}(pt;C^*_r(\Gamma))$ is well defined and fits into a commuting diagram
\[
\begin{CD}
\Omega_n^{Spin}(X)@>>>R_n^{Spin}(X)\\
@V\beta^{geo}VV @V\ind_{\Gamma}^{geo}VV \\
K_*^{geo}(X)@>\mu>>K_*^{geo}(pt;C^*_r(\Gamma)),\\
\end{CD}
\]
\end{prop}

\begin{proof}
Recall the notation $\lambda_{an}: K_*^{geo}(X,B)\to KK_*(C(X),B)$ for the analytic assembly mappings, these are isomorphisms for a finite CW-complex $X$. It follows from \cite[Theorem 1.31]{PSrhoInd} that the diagrams 
\[
\begin{CD}
\Omega_n^{Spin}(X)@>>>R_n^{Spin}(X)\\
@V\lambda_{an}\circ \beta^{geo}VV @V\lambda_{an}\circ \ind_{\Gamma}^{geo}VV \\
K_*(X)@>\mu>>K_*(C^*_r(\Gamma)),\\
\end{CD}
\]
commutes. The proof is complete (since $\lambda_{an}$ is an isomorphism).
\end{proof}

\begin{define}
For $(W,\xi,f,g)$ as in Proposition \ref{decorpsc}, we set 
\begin{align*}
\nu^{geo}(W,\xi,&f,g):=r_*(pt,\ind_{APS}(W,\Xi_g,f))\\
&-r_*(pt,\ind_{APS}(M^B,((\mathcal{L}_{M^B}^{2^k},\C^{2^k},\mathrm{id}),(D^{B},D^0),A^0),f))\in \mathcal{S}^{geo}_*(X,\mathcal{L}_X)
\end{align*}
where $r_*$ denotes the map from the $K$-theory of $C^*_r(\Gamma)$ to $\mathcal{S}^{geo}_*(X, \mathcal{L}_X)$ (see \cite[Theorem 3.8]{paperI}).
\end{define}

\begin{remark}
Despite the fact that $\nu^{geo}(W,\xi,f,g)$ is defined from a class in $K_*(C^*_r(\Gamma))$ that depends on a choice of $\Xi_g$, it follows from Proposition \ref{decorpsc} and commutativity of the left half of the diagram in Theorem \ref{themaplambda} that $\nu^{geo}(W,\xi,f,g)$ is well defined. In fact, $\lambda^\mathcal{S}_{an}\nu^{geo}(W,\xi,f,g)=\nu(W,\xi,f,g)$ (in the notation of Proposition \ref{decorpsc}). 
\end{remark}

\begin{define}
Let $\rho_\Gamma^{P,geo}:\mathrm{Pos}_n^{Spin}(X)\to \mathcal{S}^{geo}_{n+1}(X, \mathcal{L}_{X})$ denote the map defined via 
$$\rho_\Gamma^{P,geo}: (M,f,g)\mapsto\nu^{geo}(W,\xi,h,g)- [W,\xi,h],$$
where $(W,\xi,h)$ is any cycle for $\mathcal{S}^{geo}_{n+1}(X, \mathcal{L}_{X})$ such that on the level of cycles 
$$\delta(W,\xi,h)=(M^V,Q_V,f^V)-(M^V,2^k\C,f^V),$$
for some spin$^c$-vector bundle $V\to M$ of some even rank $2k$.
\end{define}
It is not clear that $\rho_\Gamma^{P,geo}$ is well-defined. First, we show that there exists a cycle in $\mathcal{S}^{geo}_{n+1}(X, \mathcal{L}_{X})$ with the property required in the definition of this map.
\begin{lemma}
Let $(M,f,g)$ be a cycle in $\mathrm{Pos}_n^{Spin}(X)$. Then, there exists $(W, \xi, h) \in  \mathcal{S}^{geo}_{n+1}(X, \mathcal{L}_{X})$ such that (at the level of cycles in geometric $K$-homology)
$$\delta(W,\xi,h)=(M^V,Q_V,f^V)-(M^V,2^k\C,f^V),$$
for some spin$^c$-vector bundle $V\to M$ of some even rank $2k$.
\end{lemma}
\begin{proof}
Given a cycle $(M,f,g)$ for $\mathrm{Pos}_n^{Spin}(X)$, Lichnerowicz formula implies that $\lambda_{an}\circ \mu_{geo}[M,\C,f]=\ind_{C^*_r(\Gamma)}(D_g)=0$, where $\mu_{geo}:K_*^{geo}(X)\to K_*^{geo}(pt,C^*_r(\Gamma))$ denotes the free assembly and $\lambda_{an}:K_*(pt;C^*_r(\Gamma))\to K_*(C^*_r(\Gamma))$ the isomorphism between realizations of the $K$-theory of $C^*_r(\Gamma)$; hence $\mu_{geo}[M,\C,f]=0$ as a class in $K_*^{geo}(pt,C^*_r(\Gamma))$. Using the notion of ``normal bordism" (see \cite{paperI, Rav}), there exists a spin$^c$-vector bundle $V\to M$ of even rank, say $2k$, a spin$^c$-manifold $W$ with boundary $M^V$ and a $K$-theory class $x\in K^0(W;C^*_r(\Gamma))$ such that 
$$(\partial W,x|_{\partial W})=(M^V,[f^*\mathcal{L}_{X}]^V)=(M^V,[\C]^V\otimes (f^V)^*\mathcal{L}_{X}),$$
as cycles with $K$-theory coefficients. For simplicity, we set $h:=f^V$. Since $[\C]^V=[Q_V]-[2^k\C]$, by definition, there exists a relative $K$-theory cocycle for assembly on $(W,\partial W,h)$ of the form $(\mathcal{E}_{C^*_r(\Gamma)},\mathcal{E}'_{C^*_r(\Gamma)},Q_V,\C^{2^k},\alpha)$ with $x=[\mathcal{E}_{C^*_r(\Gamma)}]-[\mathcal{E}'_{C^*_r(\Gamma)}]$. This completes the proof as $(W, (\mathcal{E}_{C^*_r(\Gamma)},\mathcal{E}'_{C^*_r(\Gamma)},Q_V,\C^{2^k},\alpha), h)$ has the required property.
\end{proof}
\begin{theorem}
The map $\rho_{\Gamma}^{P,geo}$ is well defined. Furthermore, $\rho_\Gamma^{P,geo}$ fits into a commuting diagram with exact rows:
\tiny
\[\begin{CD}
\ldots@>>> \Omega_n^{Spin}(X)@>>> R_n^{Spin}(X)@>>>\mathrm{Pos}_{n-1}^{Spin}(X)@>>>\Omega_{n-1}^{Spin}(X)@>>> \ldots\\
@. @V\beta^{geo} VV @V\ind_{\Gamma}^{geo} VV @V\rho_\Gamma^{P,geo}VV@V\beta^{geo} VV@. \\
\ldots  @> \delta >> K^{geo}_{n}(X) @>\mu>>K_{n}^{geo}(pt;C^*_r(\Gamma)) @>r >> \mathcal{S}^{geo}_{n}(X;\mathcal{L}_{X})  @> \delta >>K^{geo}_{n-1}(X) @>>> \ldots
\end{CD},\]
\normalsize
where the top row is Stolz' positive scalar curvature sequence and the bottom row is the geometric model for the analytic surgery exact sequence.
\end{theorem}

\begin{proof} 
To show that $\rho_{\Gamma}^{P,geo}$ is well defined, we consider $\lambda_{an}^\mathcal{S}\circ \rho_\Gamma^{P,geo}(M,f,g)\in  K_{n+1}(D^*(\tilde{X})^\Gamma)$. By the construction of $\lambda_{an}^\mathcal{S}$, \eqref{delapsforball} and Lemma \ref{rhoandvb}, 
\begin{align*}
\lambda_{an}^\mathcal{S}&\circ \rho_\Gamma^{P,geo}(M,f,g)\\
=&\bigg(\underbrace{j_{X}\ind_{APS}(W,\Xi_g,f)-\rho(D^0,A^0)}_{\nu(W,\xi,f,g)}\bigg)\\
&-\bigg(\underbrace{j_{X}\ind_{APS}(W,\Xi_{g},h)-\rho(W,\Xi_{g},h)}_{\lambda_{an}^\mathcal{S}(W,\xi,h)}\bigg)=\rho(D_g,0).
\end{align*}
It follows that, on the level of cycles, the mapping $\lambda_{an}^\mathcal{S}\circ \rho_\Gamma^{P,geo}$ coincides with the mapping $\rho_\Gamma: \mathrm{Pos}_n^{Spin}(X)\to K_{n+1}(D^*(\tilde{X})^\Gamma)$ constructed in \cite[Section 1.2]{PSrhoInd}. Since $\lambda_{an}^\mathcal{S}$ is an isomorphism, $\rho_\Gamma^{P,geo}$ is well defined because $\rho_\Gamma$ is well defined. Commutativity of the diagram follows from \cite[Theorem 1.31]{PSrhoInd} (which states that $\rho_\Gamma$ makes the analogous diagram involving the analytic surgery exact sequence commutative).
\end{proof}

\begin{theorem} \label{thmDelocPSC}
If $\Gamma$ has property $(RD)$, the following diagram commutes:
\begin{center}
$$\xymatrix{
 {\rm Pos}^{spin}_n(X) \ar[rd]_{\hat{\rho}_\Gamma}  \ar[rr]^{\rho_\Gamma^{P,geo}} &&  \mathcal{S}^{geo}_n(X;\mathcal{S}(\Gamma))
  \ar[ld]^{\ch^{del}_{\mbox{\tiny (RD)}}}   \\
  &\hat{H}^{del}_n(\Gamma) &
} $$
\end{center} 
where the map $\hat{\rho}_\Gamma$ is defined in \cite[Section 2.3]{PSrhoInd} (also see \cite{LPhigEta}) and $\mathcal{S}(\Gamma)$ denotes the Schwartz algebra. 
\end{theorem}

\section{Outlook and  mappings from the surgery exact sequence}
\label{surgsurgrem}

The terminology ``analytic surgery group" is due to Higson and Roe \cite{HRSur1,HRSur2, HRSur3} who mapped the surgery exact sequence into the analytic surgery exact sequence using Poincar\'e complexes. In recent work by Piazza and Schick \cite{PSsignInd} the same program was carried out using techniques of higher index theory. The combined outcome of the results of \cite{PSsignInd}, Theorem \ref{themaplambda} and Theorem \ref{delchethe} can be summarized in a diagram with exact rows: 
\[\scriptsize
\begin{CD}
\ldots@>>>\mathcal{N}_*(X)@>>> L_*(\Z[\Gamma]) @>>> \mathcal{S}_*(X)  @>>> \mathcal{N}_{*-1}(X)@>>>\ldots  \\
@.@V\beta VV@V\ind_\Gamma VV @V\rho VV @V\beta VV  \\
\ldots@>>>K_{*}(X)@>\mu>> K_*(C^*_r(\Gamma)) @>r >> K_*(D^*(\tilde{X})^\Gamma)  @> \delta >> K_{*-1}(X)@>>>\ldots  \\
@.@A\lambda_{an} AA@A\ind_{AS} AA @A\lambda^\mathcal{S}_{an} AA @A\lambda_{an}AA  \\
\ldots@>>>K^{geo}_{*}(X)@>\mu>> K_*^{geo}(pt;\mathcal{A}(\Gamma)) @>r >> \mathcal{S}^{geo}_*(X;\mathcal{A})  @> \delta >> K^{geo}_{*-1}(X)@>>>\ldots  \\
@.@V\ch_X^{<e>}VV@V\ch_{\mathcal{A}(\Gamma)} VV @V\mathrm{ch}^{del}VV @V\ch_{X}^{<e>}VV  \\
\ldots@>>>\hat{H}_{*}^{<e>}(\mathcal{A}(\Gamma))@>\mu_{dR}>> \hat{H}^{dR}_*(\mathcal{A}) @>r_{dR}>> \hat{H}_*^{del,<e>}(\mathcal{A}) @>\delta_{dR}>> \hat{H}_{*-1}^{<e>}(\mathcal{A})@>>>\ldots  \\
\end{CD}.\]
\normalsize
The upper part of the diagram commutes up to factors of two. The upper part of the diagram are the maps from ``surgery to analysis" in \cite{PSsignInd}. This summary of ``mapping surgery to analysis" and ``mapping geometry to analysis" indicates that one can define explicit define analytic invariants from the surgery exact sequence by ``inverting" the isomorphism from ``geometry to analysis". This process should be viewed in the context of the ``Baum approach" to index theory discussed in the introduction. In the Baum's approach, one ``inverts" the isomorphism from geometric to analytic $K$-homology.

It is also possible to map the surgery exact sequence directly into the geometric model for the analytic surgery exact sequence along the lines of Section \ref{mapstotosu}, but we will not discuss it in detail. The reader shoud note that there are two issues that make it more difficult than the case of the positive scalar curvature exact sequence. Firstly, in the surgery exact sequence case, it is more natural to work with oriented manifolds. As such, one should use a geometric model for the analytic surgery group (built on oriented manifolds) following \cite{guentnerkhom,HReta,Kescont}. Secondly, there is a technical difference: on a positive scalar curvature manifold the boundary operator is always invertible, while for the signature operator one would need to make use of the trivializing operator constructed by Hilsum-Skandalis (see \cite{PSsignInd}). The procedure of mapping ``surgery to geometry" directly circumvents inverting the mapping from ``geometry to analysis" by using higher Atiyah-Patodi-Singer index theory. At the cost of computing a higher index, this mapping provides a geometric methodology to define invariants on the surgery exact sequence  -- avoiding the analytic difficulties of delocalized $\rho$-invariants and the Higson-Roe's analytic structure group.\\

\paragraph{\textbf{Acknowledgements}}
The authors wish to express their gratitude towards Karsten Bohlen, Heath Emerson, Nigel Higson, Paolo Piazza, Thomas Schick and Charlotte Wahl for discussions. They also thank the Courant Centre of G\"ottingen, the Leibniz Universit\"at Hannover, the Graduiertenkolleg 1463 (\emph{Analysis, Geometry and String Theory}) and Universit\'e Blaise Pascal Clermont-Ferrand for facilitating this collaboration.

\end{document}